\documentclass[reqno,12pt]{amsart}
\usepackage{xr}
\usepackage{graphicx}
\usepackage[usenames,dvipsnames,svgnames,table]{xcolor}
\usepackage{amsrefs}

\usepackage{epstopdf}
\usepackage{booktabs}
\usepackage{amssymb,amsmath}
\usepackage{amsthm}
\usepackage{enumerate}

\newtheorem{theo}{Theorem}[section]

\newtheorem{defi}[theo]{Definition}

\newtheorem{prop}[theo]{Proposition}

\newcommand{\cgot}{\ensuremath{\mathfrak{c}}}
\newcommand{\ggot}{\ensuremath{\mathfrak{g}}}
\newcommand{\hgot}{\ensuremath{\mathfrak{h}}}
\newcommand{\kgot}{\ensuremath{\mathfrak{k}}}

\newcommand{\tgot}{\ensuremath{\mathfrak{t}}}

\newcommand{\Acal}{\ensuremath{\mathcal{A}}}
\newcommand{\Bcal}{\ensuremath{\mathcal{B}}}
\newcommand{\Ccal}{\ensuremath{\mathcal{C}}}
\newcommand{\Dcal}{\ensuremath{\mathcal{D}}}

\newcommand{\Fcal}{\ensuremath{\mathcal{F}}}

\newcommand{\Ical}{\ensuremath{\mathcal{I}}}
\newcommand{\Kcal}{\ensuremath{\mathcal{K}}}
\newcommand{\Lcal}{\ensuremath{\mathcal{L}}}

\newcommand{\Ncal}{\ensuremath{\mathcal{N}}}
\newcommand{\Ocal}{\ensuremath{\mathcal{O}}}

\newcommand{\Qcal}{\ensuremath{\mathcal{Q}}}
\newcommand{\Rcal}{\ensuremath{\mathcal{R}}}

\newcommand{\Scal}{\ensuremath{\mathcal{S}}}

\newcommand{\Ucal}{\ensuremath{\mathcal{U}}}

\newcommand{\Ycal}{\ensuremath{\mathcal{Y}}}

\newcommand{\Z}{\ensuremath{\mathbb{Z}}}
\newcommand{\C}{\ensuremath{\mathbb{C}}}

\newcommand{\R}{\ensuremath{\mathbb{R}}}

\newcommand{\N}{\ensuremath{\mathbb{N}}}

\newcommand{\T}{\ensuremath{\hbox{\rm T}}}

\newcommand{\Qfor}{\ensuremath{\Qcal^{-\infty}}}

\newcommand{\Rfor}{\ensuremath{\hat R}}

\newcommand{\Ker}{\ensuremath{Ker}}

\newtheorem{theorem}{Theorem}[section]
\newtheorem{proposition}[theorem]{Proposition}
\newtheorem{lemma}[theorem]{Lemma}
\newtheorem{definition}[theorem]{Definition}

\newtheorem{example}[theorem]{Example}
\numberwithin{equation}{section}

\newcommand{\A}{{\mathcal{A}}}

\newcommand{\CB}{{\mathcal{B}}}
\newcommand{\CL}{{\mathcal{L}}}
\newcommand{\CO}{{\mathcal{O}}}

\renewcommand{\ll}{{\langle}}
\newcommand{\rr}{{\rangle}}

\def \clif {{\bf c}}

\newcommand{\qfor}{\ensuremath{\mathcal{Q}^{-\infty}}}

\newcommand{\g}{\mathfrak{g}}
\newcommand{\h}{\mathfrak{h}}
\renewcommand{\t}{\mathfrak{t}}
\renewcommand{\q}{\mathfrak{q}}
\newcommand{\Ahat}{{\hat{A}}}

\newcommand{\res}{\ensuremath{\hbox{\rm rescaling}}}

\newcommand{\QS}{\ensuremath{\Qcal^{\mathrm{spin}}}}

\newcommand{\vol}{\ensuremath{{\rm vol}}}

\def \wG {{\widehat{G}}}
\def \wT {{\widehat{T}}}
\def \wH {{\widehat{H}}}

\begin{document}
\title[Dirac operators and semi-classical limits.]{The equivariant index of twisted Dirac operators and semi-classical limits.}
\date{}

\author{Paul-Emile PARADAN  and
Mich\`ele VERGNE}

\begin{abstract}
Let $G$ be a compact connected Lie group with Lie algebra $\g$.
 Let $M$ be a compact spin manifold with a $G$-action, and $\CL$ be a $G$-equivariant line bundle on $M$.
 Consider an integer $k$, and let
  $\QS_G(M,\CL^k)$ be the equivariant index of the Dirac operator on $M$
 twisted by $\CL^k$.
 Let  $m_G(\lambda,k)$ be the multiplicity
  in  $\QS_G(M,\CL^k)$  of the irreducible representation of $G$ attached to the admissible coadjoint orbit $G\lambda$.
 We prove that the distribution
  $\ll \Theta_k,\varphi\rr= k^{\dim(G/T)/2}\sum_{\lambda} m_G(\lambda,k)\ll\beta_{\lambda/k},\varphi\rr $ has an asymptotic expansion
when $k$ tends to infinity of the form
$\ll\Theta_k,\varphi\rr\equiv k^{\dim M/2}\sum_{n=0}^{\infty} k^{-n} \ll \theta_n,\varphi\rr$.
Here $\varphi$ is a test function on $\g^*$ and $\ll\beta_\xi,\varphi\rr$ is the integral of $\varphi$ on
the coadjoint orbit $G\xi$ with respect to the canonical Liouville measure.
We compute explicitly the distribution $\theta_n$ in terms of the graded $\hat A$ class of $M$ and the equivariant curvature of $\CL$.

If $M$ is non compact, we use these asymptotic techniques to give another proof of the
fact that the formal geometric quantification of a manifold with a spinc structure is functorial with
respect to restriction to subgroups.

\end{abstract}

\maketitle

{\small
\tableofcontents}

{\bf To the memory of Bertram Kostant.}

\medskip

This article is pursuing the fundamental idea of Kostant that  a line bundle $\CL$ on a $G$-manifold $M$,
equipped with a $G$-invariant connection, give rise
 to a moment map $\phi_G: M\to \g^*$, and thus  hopefully to a relation between
the quantization of $M$ and the representation theory of $G$,  since $\wG$ may be considered
as a subset of $\g^*$. We investigate further the corresponding quantization, via Dirac operators, in the case
where the group $G$ is compact, and the moment map proper. Our new insight is that the asymptotic
behavior when $\CL$ is replaced by $\CL^k$ is easy to compute classically, and determines completely the quantization at $k=1$. Certainly, Bertram Kostant would have found this idea obvious, but maybe also beautiful. Anyway, here it is.

\section{Introduction}

Let $M$  be a compact oriented spin manifold  of even  dimension $2d$.
 Let $\CL$ be a line bundle  over $M$ equipped with a connection of curvature $-i\Omega$ and let $\hat A (M)$ be the $\hat A$ class of $M$ (normalized as in \cite{BGV}).
 We do not assume the curvature $-i\Omega$ of $\CL$ to be non-degenerate.
Define the integral
$$\Qcal^{\rm geo}(M,\CL)= \frac{1}{(-2i\pi)^d} \int_M e^{-i\Omega} {\hat A}(M).$$
The Dirac operator $D_\CL$ twisted by $\CL$ is an elliptic operator on $M$, and let $\QS(M,\CL)=\dim {\rm \Ker}(D_{\CL})-\dim {\rm Coker}(D_{\CL})$ the index of $D_\Lcal$.
By the Atiyah-Singer index theorem,
$\QS(M,\CL)=
\Qcal^{\rm geo}(M,\CL).$

Assume now that a compact connected Lie group $G$, with Lie algebra $\g$, acts on $(M,\CL)$.
The choice of an invariant connection on the line bundle $\CL$
determines a moment map $\phi_G: M\to \g^*$.
This is the Kostant moment map  \cite{Kostant}.
Let $T$ be a maximal torus of $G$, $\t$ its Lie algebra.
We identify the space $\wG$ of irreducible finite dimensional representations of $G$
 to the discrete set of elements $\lambda\in  \t^*$,
which are dominant admissible and regular, and
we denote by $\chi_\lambda(g)$ the trace of the action of $g\in G$
 on the irreducible representation of $G$  parameterized by
$\lambda$.

If $k$ is an integer, we denote by $\CL^k$ the $k$-th power of the line bundle $\CL$.
Define  the  function $\QS_{G}(M,\CL^k)$  on $G$ to be the trace of the action of $g\in G$
in the virtual space   ${\rm \Ker}(D_{\CL^k})-{\rm Coker}(D_{\CL^k})$.
 Define $m_{G}(\lambda,k)\in \Z$ such that
$$\QS_{G}(M,\CL^k)=
\sum_{\lambda\in {\hat G}} m_{G}(\lambda,k)\chi_\lambda.$$

Consider the geometric (re-scaled) analogue
$$
\Theta^{(M,\CL)}_k=  k^{\frac{\dim(G/T)}{2}}\sum_{\lambda \in {\hat G}} m_G(\lambda,k)\beta_{\lambda/k},
$$
the weighted sum of  the canonical Liouville measures  $\beta_{\lambda/k}$
on the coadjoint orbits $G\lambda/k$.

The aim of this article is to study the asymptotic behavior of
$\Theta^{(M,\CL)}_k$ when $k$ is large, and $M$ possibly non compact,
and to explore  a functorial consequence of this formula for
 reduced spaces.

We work in the spin context. The same argument would adapt to manifolds with spinc structures, provided we work with odd powers of $k$.
This context is more general than the Hamiltonian context (which includes
 the K\"{a}hler context), and it is more natural since Weyl character formula for  $\chi_\lambda$  is the fixed point formula for a  twisted Dirac operator on the coadjoint orbit $G\lambda$.
 We will return to the comparison between both contexts in a forthcoming article.
The article \cite{ver:graded-CRAS} determines the  asymptotic expansion of $\Theta^{(M,\CL)}_k$
in the Hamiltonian context, when $G$ is a torus.

Let us return  to the case of a  spin manifold $M$ with a $G$-invariant line bundle with connection. We do not assume 
anymore  $M$ compact, but we assume the moment map $\phi_G$ proper.
In this introduction,  to simplify, assume that all stabilizers of the action of $G$ on $M$ are abelian and connected.
We also assume that $\rho$ is in the weight lattice of $T$. For every $\lambda\in \wG\subset \wT\subset \t^*$, consider the reduced space
$$
M_{\lambda, G}:=\phi_G^{-1}(\lambda)/T
$$
which may be non connected. If $\lambda$ is a regular value of $\phi_G$, $M_{\lambda, G}$
is a smooth compact spin manifold.
Denote by $\C_{[-\lambda]}$ the vector space $\C$ with the action of $T$ given by the character $t^{-\lambda}$  of $T$.
The restriction  of  the line bundle $\C_{[-\lambda]}\otimes \CL^k$ to
 $\phi_G^{-1}(\lambda)$ is the pull back of a line bundle $\CL_{\lambda,k}$
 on $M_{\lambda, G}$.
So define
$$
m^{\rm geo}_G(\lambda,k)=\Qcal^{\rm geo}(M_{\lambda, G}, \CL_{\lambda,k}).
$$
 If $\lambda$ is not a regular value of $\phi_G$, this number can be defined by deformation.
  When $M$ is compact, the $[Q,R]=0$ theorem in the spin context \cite{pep-vergne:acta} asserts that $m_G(\lambda,k)=m^{\rm geo}_G(\lambda,k)$.
This suggests to extend the definition of $\QS_{G}(M,\CL)$,
when $\CL$ is equipped with a connection (that we leave implicit in the notation)
such that the moment map $\phi_G$ is proper, in the following way. Define
$$\Qcal^{\rm geo}_G(M,\CL^k):=\sum_{\lambda} m^{\rm geo}_G(\lambda,k) V_\lambda.$$
We call this object (in the spirit of \cite{Weitsman}) the formal geometric quantization of $(M,\CL)$.

Here are two simple examples.

\begin{example}

$\bullet$ $M=T^*S^1$, with coordinates $(t,\theta)$,
and the natural action of the circle group $G=S^1$.
 Consider the trivial line bundle
$\CL$ with connection $d-itd\theta$.
The moment map is $\phi_G(t,\theta)=t$.
Thus  $\Qcal^{\rm geo}_G(M,\CL^k)$ is independent of $k$:
$$\Qcal^{\rm geo}_G(M,\CL^k)=\sum_{n\in \Z} e^{in\theta}.$$

$\bullet$  $M=\C$, with coordinate $z$.
 We consider an integer $a$ and the  line bundle $\CL=M\times \C$
 with action $e^{i\theta}(z,v)=(e^{2i\theta}z, e^{ia\theta}v)$.
 Take the connection $\nabla=d-\frac{i}{2} Im (zd\bar{z})$.
The moment map is  $\phi_G(z)=a+|z|^2 $ and is proper.
Then
$$\Qcal^{\rm geo}_{G}(M,\CL^k)=e^{ika\theta}
\sum_{j\geq 0 } e^{i(2j+1) \theta}.$$
\end{example}

As $M$ is non compact, it is not possible to define the index of $D_\CL$ without introducing additional data.  
We use transversally elliptic operators. The Kirwan vector field $\kappa_G$  associated to $\phi_G$ allows us to deform 
the symbol of the Dirac operator, and to obtain a $G$-invariant transversally elliptic operator $D_{\CL,\phi_G}$ on $M$ 
if the set of zeroes of $\kappa_G$ is compact, which we assume in this introduction.  The index $\Qfor_G(M,\CL^k)$ of this 
operator  provides a well defined generalized function on $G$. So
 define $m_{G}(\lambda,k)\in \Z$ such that
$$
\Qfor_{G}(M,\CL^k)(g)=
\sum_{\lambda\in {\hat G}} m_{G}(\lambda,k)\chi_\lambda(g).
$$
Then, we have again \cite{Hochs-Song:duke}
$$m_G(\lambda,k)=m^{\rm geo}_G(\lambda,k).$$
An important consequence of this geometric relation is
the fact that the function $m_G(\lambda,k)$ is a piecewise quasi-polynomial function.
In particular, the map $k\mapsto m_G(k\lambda,k)$ is entirely determined by its large behavior.

We  define as in the compact case the weighted sum of  measures of the orbits $G\lambda/k$:
$$\Theta^{(M,\CL)}_k=k^{\dim(G/T)/2}\sum_{\lambda} m_G(\lambda,k)\beta_{\lambda/k}.$$

As it is well known, at least in the Hamiltonian case, the measure
$\frac{1}{k^{\dim M/2}}\Theta^{(M,\CL)}_k$ tends  to the Duistermaat-Heckman measure when $k$ tends to $\infty$.
 Here we will prove that
$\Theta^{(M,\CL)}_k$ has an asymptotic expansion, when $k\to \infty$,
  as a Laurent series (in $1/k$) of distributions, and we will determine all coefficients
as twisted Duistermaat-Heckman distributions related to the $\hat A$ class of $M$.

Recall in this spin context  the definition of the Duistermaat-Heckman measure.
 Let $\Omega(X)=\Omega-\langle \phi_G,X\rangle$, $X\in \g$.
This is a closed equivariant form  of degree $2$ on $M$ and $e^{-i\Omega(X)}$
is the equivariant Chern character of $\CL$. The form $\Omega^{\dim M/2}$
is a density on $M$ (which might  be not positive).
The Duistermaat-Heckman measure is the signed measure on $\g^*$ obtained by push-forward of  $\Omega^{\dim M/2}$ by the proper map $\phi_G$.
If $M$ is compact, its Fourier transform is the function
$X\mapsto \frac{1}{(-2i\pi)^d}\int_M e^{-i\Omega(X)}$ on $\g$.
Similarly, if $\nu(X)$ is a closed $G$-equivariant form on $M$, with polynomial coefficients, we can define  the distribution $DH^G(M,\Omega,\nu)$ on $\g^*$ by the formula:
\begin{equation}\label{DHmu}
\ll DH^G(M,\Omega,\nu),\varphi\rr=\frac{1}{(-2i\pi)^d}\int_{M}\int_{\g} e^{-i\Omega(X)}
\nu(X) \widehat{\varphi}(X)dX.
\end{equation}
Here  $\varphi$ is a test  function on $\g^*$, with Fourier transform
$\widehat{\varphi}$.
%
It is easy to see that this distribution is well defined if $\phi_G$ is proper.

Consider the equivariant $\hat A$  class  of $M$ which belongs to the completion of the equivariant cohomology ring of $M$, and
its expansion $\hat A(M)(X)=\sum_{n=0}^{\infty} \Ahat_n(M)(X)$ in equivariant classes homogeneous of degree $2n$.
Finally,  let $j_{\g}^{1/2}(X)=\det_{\g}^{1/2}\left(\frac{e^{X/2}-e^{-X/2}}{X}\right)$,
  a $G$-invariant function of  $X\in \g$.
  It determines a formal series $j_\g^{1/2}(i\partial/k)$
  of $G$-invariant constant coefficient differential operators on $\g^*$.

Our main contribution is the following result (see Theorem \ref{asympt1-noncompact}).

\begin{theorem}\label{theo:intro}
When $k$ tends to $\infty$,
\begin{equation}\label{eq:theo-intro}
\Theta^{(M,\Lcal)}_k \equiv
 j_\g^{1/2}(i\partial/k)\Big(k^d\sum_{n=0}^{\infty}
   \frac{1}{k^n} DH^G(M,\Omega, \Ahat_n(M))\Big).
\end{equation}
   \end{theorem}

The leading term  in $k^d$ is the Duistermaat-Heckman measure.

\bigskip

Let us check this formula in our two basic examples.

\begin{example}\label{ex:cohen}

$\bullet$ For $M=T^*S^1$, the $\hat A$ class is identically equal to $1$, so there exists only one term in the expansion (\ref{eq:theo-intro}).
Here Theorem \ref{theo:intro} is the well known fact that
$$\sum_{\lambda\in \Z} \varphi(\lambda/k)\equiv k \int_\R\varphi(\xi) d\xi.$$

$\bullet$
For $M=\C$, the equivariant $\hat A$ class for our action is the invariant function
$\frac{\theta}{\sin(\theta)}$.
Theorem \ref{theo:intro}
is the following variation of the Euler-MacLaurin formula (see \cite{cohen})
$$\sum_{j\geq 0}\varphi(a+(2j+1)/k)\equiv
\frac{k}{2}\int_{a}^{\infty}
 \varphi(\xi)d\xi-
\sum_{n=1}^{\infty}(\frac{2}{k})^{n-1} \frac{B_{n}(1/2)}{n!}\varphi^{(n-1)}(a).$$
Here $B_{n}(t)$ is the $n$-th Bernoulli polynomial
and only even integers $n$ occur,
since $B_{n}(\frac{1}{2})=0$ for all odd $n\geq 0$.

\end{example}

\medskip

Let us see why this formula is natural, when $M$ is compact.
By Kirillov formula, the Fourier transform of
$\Theta^{(M,\Lcal)}_k$  is the function
$$X\longmapsto \QS_G(M,\CL^k)(\exp X/k)j_\g^{1/2}(X/k)$$ on $\g$.
The delocalized formula \cite{BGV}
for the index says that, for $X\in \g$ small enough,
$$
\QS_G(M,\CL^k)(\exp X/k)=\frac{1}{(-2i\pi)^d} \int_M e^{-ik\Omega(X/k)} {\hat A}(M)(X/k).
$$
But, by simple inspection of the relation between the equivariant degree and the polynomial degree, this is also equal for $X$ small to
$$\frac{1}{(-2i\pi)^d} k^d\int_M e^{-i\Omega(X)}\Big(\sum_{n=0}^{\infty}\frac{1}{k^n}\Ahat_{n}(M)(X)\Big).$$
The formula for the asymptotic expansion of
$\Theta^{(M,\Lcal)}_k$ follows easily.
Remark the dichotomy between the equivariant Chern character
$e^{-i\Omega(X)}$
 that is unchanged in this  asymptotic equivalence,
and involves the exponential function $e^{i\ll\phi_G,X\rr}$, and
the $\hat A$ class that we expand as a series of homogeneous equivariant classes
with polynomial coefficients.
A more delicate analysis, using the delocalized formula (\cite{pep-vergne:bismut})
for the generalized function $\Qfor_G(M,\CL^k)(\exp X)$ leads us also naturally to Theorem \ref{theo:intro}.
The Chern character  $e^{-i\Omega(X)}$ has to be replaced
by the Chern character  $e^{-i\Omega(X)}P(X)$ where $P(X)$ is an equivariant form with
generalized coefficients, supported in a compact neighborhood of the zeroes of $\kappa_G$, equivalent to $1$ in the equivariant cohomology without compact support conditions.

 In conclusion, Theorem  \ref{theo:intro}  says that the
 formula
 $$\QS_G(M,\CL)(\exp X)=\frac{1}{(-2i\pi)^d}
 \int_M  e^{i\Omega(X)}\Ahat(M)(X)$$
 has a meaning in the asymptotic sense for a non compact manifold $M$ when $\CL$ is replaced by $\CL^k$ and $X$ by $X/k$, and is a good ersatz for  the
 equivariant index formula of $D_\CL$.

\bigskip

Let us now explain a consequence of this asymptotic formula for
reduced spaces.

Return for a moment to the case where $M$ is compact.
Plugging $g=1$ in the identity
$\QS_{G}(M,\CL^k)(g)=
\sum_{\lambda\in {\hat G}} m_{G}(\lambda,k)\chi_\lambda(g)$,
leads to the remarkable identity
\begin{equation}\label{eq:mystery}
\Qcal^{\rm geo}(M,\CL^k)=\sum_{\lambda\in \widehat{G}}\,\vol(G\lambda)\,
\Qcal^{\rm geo}(M_{\lambda, G},\CL_{\lambda,k})
\end{equation}
relating an integral on $M$ to a  sum of integrals on the finite number of reduced spaces $M_{\lambda,G}$.
We will see that this equality generalizes to reduction in stages,
even when $M$ is non compact.

Let $H$ be a connected compact subgroup of $G$ with torus $T_H$.
The $H$ action on $(M,\CL)$
leads to the moment map $\phi_H: M\to\h^*$.
Assume $\phi_H$  is still proper.
 In this case we can define $\Qfor_G(M,\Lcal)$, and
 $\Qfor_H(M,\Lcal)$.
When $M$ is compact, from the description of
$\QS_G(M,\CL)$ as the index of  the elliptic operator $D_{\CL}$ on $M$,
we see that  $\QS_H(M,\Lcal)$  is the restriction of the representation $\QS_G(M,\CL)$ to $H$.
When $M$ is non compact, this relation is not obvious
 since our $G$-transversally elliptic operator 
 $D_{\CL,\phi_G}$
 is not (usually) $H$-transversally elliptic.
The following theorem
\begin{equation}\label{eq:formal}
 \Qfor_G(M,\Lcal)\vert_H=\Qfor_H(M,\Lcal)
 \end{equation}
 was  proved by the first author  using cutting \cite{pep:formal3}.
Here we will show that this also follows from comparing
 the large  behavior of $\Qfor_G(M,\Lcal^k)$ and $\Qfor_H(M,\CL^k)$
 and the fact that multiplicities $m_G(k\lambda,k), m_H(k\lambda,k)$ are
 entirely determined by their large values.

 Let $\lambda\in \wG$, and $\mu\in \wH$.
Denote by $c(\lambda,\mu)$ the multiplicity of $V_\mu$ in $V_\lambda\vert_H$.
 Theorem (\ref{eq:formal}) is equivalent to the following equality
of indices of Dirac operators on reduced spaces for $H$, and $G$.

 For any $\mu\in \wH$, we have
$$
\Qcal^{\rm geo}(M_{\mu,H},\CL_{\mu,k})=
\sum_{\lambda\in \wG}\, c(\lambda,\mu) \,\Qcal^{\rm geo}(M_{\lambda,G},\CL_{\lambda,k}).
$$

When $M$ is compact, we can take $H=\{1\}$,
and this is the mysterious equality (\ref{eq:mystery}).
In general, this equality of indices of Dirac operators on reduced spaces for different groups is not clear, since
$M_{\mu,H}=\phi_H^{-1}(\mu)/T_H$ do not carry  any visible group action.

Let us briefly sketch a proof of this  equality for the case
where $G,H$ are  torus.
 This simple case could be  treated   by considering the action of $G/H$ on the compact spin manifold $M_{\mu,H}$, 
 but we treat it by a different method which will generalize to any pair of groups $H\subset G$.

  In this case,  $\wG$, $\wH$ are lattices in $\g^*,\h^*$,
  and we have to prove
  $$m_H(\mu,k)=\sum_{\lambda \in \wG, \lambda|_{\h}=\mu}m_G(\lambda,k).$$
  In a companion article \cite{pep-vergne:asymptotics}, we proved the following easy result.
Consider
a distribution
$D_k=\sum_{\mu\in \wH} q(\mu,k)\delta_{\mu/k}$
associated to a piecewise quasi-polynomial function $q(\mu,k)$.
Assume that, for any $\zeta\in H$, an element of finite order,
  the distribution
$\sum_{\mu\in \hat H} q(\mu,k)\zeta^\mu \delta_{\mu/k}=O(k^{-\infty}) $
when $k$ tends to $\infty$. Then $q(\mu,k)=0$.

We have computed the asymptotic expansion of both distributions  $\Theta_k^G= \sum_{\lambda \in \hat G} \, m_G(\lambda,k)\,\delta_{\lambda/k}$ and 
$\Theta_k^H = \sum_{\mu\in \hat H} \, m_H(\mu,k)\,\delta_{\mu/k}$. One has:
\begin{eqnarray*}
\Theta_k^G &\equiv &  k^{d}\,\sum_{n=0}^{\infty}\frac{1}{k^n} DH^G(M,\Omega,{\hat A}_n(M)),\\
\Theta_k^H &\equiv & k^{d}\,\sum_{n=0}^{\infty}\frac{1}{k^n} DH^H(M,\Omega,{\hat A}_n(M)).
\end{eqnarray*}

Define $m_H'(\mu,k)=\sum_{\lambda \in \wG, \lambda|_{\h}=\mu}m_G(\lambda,k)$.
The distribution
$$S_k^H=\sum_{\mu\in \hat H} m_H'(\mu,k)\delta_{\mu/k}$$
is the push-forward $r_*\Theta_k^G$ of the distribution
$\Theta_k^G$ under the map $r:\g^*\to \h^*$.

The Duistermaat-Heckman distributions
 $DH^G(M,\Omega,\nu)$ behave very well under the push-forward map:
  $$r_*DH^G(M,\Omega,\nu)=DH^H(M,\Omega,\nu).$$

  Indeed, at least in the compact case, the Fourier transform of
\break  $DH^H(M,\Omega,\nu)$ is the restriction to $\h$ of the function
 $ \frac{-1}{(2\pi)^d}
 \int_M e^{-i\Omega(X)}\nu(X)$ on $\g$.
So we see  that $S_k^H$ and $\Theta_k^H$ have the same asymptotic expansion.
This is not entirely sufficient to prove that $\Theta_k^H=S_k^H$.
But we use the fact that both functions $m_H'(\mu,k)$ and $m_H(\mu,k)$
are piecewise quasi-polynomials and that a similar asymptotic descent formula holds
for the distribution
$$
\sum_{\mu\in \wH} m_H(\mu,k) \zeta^{\mu}\delta_{\mu/k}
$$
when $\zeta\in T_H$ is of finite order.

The proof, sketched here for $H$ and $G$ abelian, works equally well for two subgroups $H\subset G$.
We just have to use the formula for the push-forward of the Liouville measure $r_*(\beta_\lambda)$ of admissible coadjoint orbits.
As expected, the full  series for the $\hat A(M)$ class, as well
as the Duflo operator $j_\g^{1/2}(\partial)$  plays a role
in the functoriality.

\bigskip

Asymptotic behavior of quantization when $k$ tends to $\infty$ has
been considered by many authors.
Let us give a very limited  and personal selection of influential works.

Asymptotic behavior of kernels of
Laplacians twisted by $\CL^k$
were used  by Boutet de Monvel-Guillemin
\cite{bou} to produce a formal star-product of functions  on symplectic manifolds.
In general only the few first  terms of the star product formal deformation are computable (see  for example \cite{Mei-toeplitz},  \cite{mama}).

Our asymptotic  trace formula for the transversally elliptic operator $D_{\CL,\phi_G}$ has the same flavor
than Fedosov trace formula  for deformation quantization \cite{fedosov}.

The article of E. Meinrenken \cite{Mei-RR}
where  multiplicities $m_G(\lambda,k)$ are identified at the large limit to  index of  reduced spaces  is in close relation with our setting.
Asymptotic Riemann sums of values of smooth functions at equally spaced sample points of a Delzant polytope $\Delta\subset\g^*$ were studied by Guillemin-Sternberg \cite{guisteriemann}, where the full asymptotic formula is given in terms of  the Todd class of the corresponding toric manifold.
The natural geometric re-scaling $\mu\to \mu/k$ in the computation of $\frac{1}{k^d}\sum_{\mu\in\Z^d\cap k\Delta}\varphi(\mu/k)$
leads to consider re-scaling $X\to X/k$ in the variable $X\in \g$.
This was one of the inspiring examples.

Our application to restrictions to subgroups is a striking example where an obvious property at the semi-classical level (functoriality of push-forward of distributions)
can lead to a proof of a subtle relation at the quantum level.

\section{Asymptotics and representations}

\subsection{Fourier transforms}\label{sec:fourier}

When $V$ is a finite dimensional real vector space, we denote by $\Dcal'(V^*)$ the vector space of complex valued distributions on the dual vector space $V^*$. If $A\in \Dcal'(V^*)$ and
$\varphi\in \Ccal^{\infty}_{cpt}(V^*)$, we denote by $\ll A,\varphi\rr \in \C$ (or $\ll A(\xi),\varphi(\xi)\rr $) the value of   $A$ on $\varphi$.

When $A\in \Dcal'(V^*)$ has a compact support, its Fourier transform $\Fcal_{V^*}(A)\in \Ccal^{\infty}(V)$ is defined by the relation
$$
\Fcal_{V^*}(A)(v):=\ll A(\xi), e^{i\langle\xi,v\rangle}\rr ,\quad v\in V.
$$

When $f$ is a smooth function on $\g$ and has at most polynomial growth, its Fourier transform $\Fcal_{V}(f)\in\Dcal'(V^*)$  is defined by the relation
$$
\ll \Fcal_{V}(f),\varphi\rr = \int_{V} f(v)\widehat{\varphi}(v)dv
$$
where $\widehat{\varphi}(v)=\frac{1}{(2\pi)^{\dim V}}\int_{V^*}\varphi(\xi)e^{-i\langle\xi,v\rangle}d\xi$. Here $dv$ and $d\xi$ are dual volume forms on $V$ and $V^*$.

For any $k\geq 1$, we define the operation
\begin{equation}
\res_{V^*}(k): \Dcal'(V^*)\to \Dcal'(V^*)
\end{equation}
by $\ll \res_{V^*}(k)(A), \varphi\rr =\ll A(\xi),\varphi(\xi/k)\rr $. We have also the operation
$\res_{V}(k):\Ccal^{\infty}(V)\to \Ccal^{\infty}(V)$ defined by $\res_{V}(k)(f)(v)= f(v/k)$.

\subsection{The isomorphism $\Rcal_\g$}\label{sec:iso-R-g}

Let $G$ be a connected compact Lie group with Lie algebra $\g$. Let $T\subset G$ be  a maximal torus with Lie algebra $\t$. Let $W_G$ be the Weyl group.
We consider $\wT$ as a lattice $\Lambda$ in $\t^*$.
If $\lambda\in \Lambda$, we denote by  $t^{\lambda}$
the corresponding character of $T$. If $t=\exp (X)$ with $X\in \t$, then
$t^{\lambda}=e^{i\ll\lambda,X\rr}$.
We choose a system $\Delta^+$ of positive roots. In our convention, $\Delta^+$ is contained in $\t^*$.
We choose an invariant scalar product on $\g$.

Let $\Dcal'(\ggot^*)^G$ be the space of $G$-invariant distributions on $\ggot^*$ and
$\Dcal'(\tgot^*)^{W_G-alt}$ be the subspace of $W_G$ anti-invariant distributions on $\t^*$.
Let $\Pi_{\g/\t}(X)=\prod_{\alpha\in \Delta^+}\langle\alpha,X\rangle$, a $W_G$ anti-invariant function on $\t$.

We will use the following classical fact.
\begin{lemma}
There exists a unique linear isomorphism
$$
\Rcal_\ggot: \Dcal'(\ggot^*)^G\longrightarrow \Dcal'(\tgot^*)^{W_G-alt}
$$
with the following two properties.
\begin{itemize}
\item $\Rcal_\ggot(f A)=f\vert_{\tgot^*} \Rcal_\ggot(A)$ for all $A\in \Dcal'(\ggot^*)^G$ and $f\in \Ccal^{\infty}(\ggot^*)^G$.
\item For compactly supported distributions $A$, the Fourier transform of $\Rcal_\ggot(A)$  is given by
$$
\ll \Rcal_\ggot(A), e^{i\langle - ,X\rangle}\rr =  (i)^r\Pi_{\g/\t}(X) \ll A, e^{i\langle - ,X\rangle}\rr ,\qquad X\in\tgot,
$$
with $r=\dim(G/T)/2$.
\end{itemize}
\end{lemma}

\subsection{Asymptotics}\label{sec:asymptotics}

Let $\Dcal'(\ggot^*)^G$ be the space of $G$-invariant distributions on $\ggot^*$.
We will  study asymptotic expansions of sequences of distributions.

\begin{definition}Let $(\Theta_k)_{k\geq 1}$ and $(\theta_n)_{n\in \N}$ be two sequences in $\Dcal'(\ggot^*)^G$. We write
\begin{equation}\label{eq:asymptotics:definition}
\Theta_k\equiv k^{n_o}\sum_{n=0}^{\infty} \frac{1}{k^n}\theta_n
\end{equation}
for some $n_o\in\Z$ if for any test function $\varphi$ and any $N\in \N$,  we have
$$
\ll \Theta_k,\varphi\rr = k^{n_o}\sum_{n=0}^N \frac{1}{k^n}\ll \theta_n, \varphi\rr  + o(k^{n_o-N}).
$$
\end{definition}

In this text, we will allow each distribution $\theta_n$ to depend periodically of the integers $k$. This means that, for each $n$, 
there exists $D\geq 1$ and distributions $(\theta_\zeta)$ parameterized by the roots of unity $\{\zeta^D=1\}$ such 
that $\theta_n(k)=\sum_{\zeta^D=1} \zeta^k \theta_\zeta$ for all $k\geq 1$.

\begin{example}\label{ex:EML}
Consider the distributions ${\rm T}_k=\sum_{\lambda\in \N}\delta_{\lambda/k}$ ($k\geq 1$) on $\R$.
The Euler Maclaurin formula gives the expansion
$$
{\rm T}_k\ \equiv \ k\, 1_{[0,\infty[} + \frac{1}{2}\delta_0 - \sum_{n=1}^{\infty} \frac{1}{k^{2n-1}}\frac{b_{2n}}{(2n)!}\delta_0^{(2n-1)},
$$
where $b_{2n}$ are the Bernoulli numbers and $\delta_0^{(2n-1)}$ is the $(2n-1)$-th derivation of the Dirac distribution
$\delta_0$.

\end{example}

Let $d$ be a $C^{\infty}$ function  defined near the origin $0\in \g$.
 Consider the Taylor series
 $\sum_{n=0}^{\infty}d_n(X)$ of $d$ at the origin.
 Thus, $d_n$ is an homogeneous polynomial function of degree $n$ on $\ggot$.
We associate to the function $d$ the formal series
$$d(i\partial/k)=\sum_{n=0}^{\infty} k^{-n}d_n(i\partial)$$ of constant coefficient
differential operators $d_n(i\partial)$ on $\ggot^*$. Thus, if
$Q(k)=k^{n_o}\sum_{n=0}^{\infty} k^{-n} \theta_n$ is a formal series of distributions
on $\g^*$, the series $d(i\partial /k) Q(k)$ is the formal series of distributions
$k^{n_o}\sum_{n=0}^{\infty} k^{-n}s_n$ on $\g^*$, where
$s_n=\sum_{l+m=n}d_l(i\partial)\theta_m$.

\bigskip

We have the following  basic fact.

\begin{lemma}
The expansion $\Theta_k\equiv k^{n_o}\sum_{n=0}^{\infty} \frac{1}{k^n}\theta_n$ holds in $\Dcal'(\ggot^*)^G$ if and only if we have the
expansion $\Rcal_\g(\Theta_k)\equiv k^{n_o}\sum_{n=0}^{\infty} \frac{1}{k^n}\Rcal_\g(\theta_n)$ in
$\Dcal'(\tgot^*)^{W_G-alt}$.
\end{lemma}

Our asymptotic expansions will arise from Taylor series.

Let $\tau$ be a  smooth  function on $\g$.
Assume  that $\tau$ and all its derivatives have at most polynomial growth.
 Consider the Taylor  series $\sum_{n=0}^\infty\tau_n(X)$ of $\tau$ at the origin.
If $b(X)$ is a smooth function on $\g$ with at most polynomial growth, then
$f_k(X)=b(X)\tau(X/k)$ ($k\geq 1$)
defines a tempered distribution on $\g$.
Similarly,  as $\tau_n(X)$ is a polynomial function,
$g_n(X)=b(X)\tau_n(X)$ defines a tempered distribution on $\g$.
We note the following result.
\begin{proposition}\label{prop:technical-asymptotics}
When $k$ tends to $\infty$,  we have  the asymptotic expansion
$$
\Fcal_\g(f_k)\equiv \sum_{n=0}^{\infty} \frac{1}{k^n}\Fcal_\g(g_n) =\tau(i\partial/k)\Fcal_\g(b).
$$
\end{proposition}
Later, we will need a uniform  version with parameters of this proposition. So we give the proof.
\begin{proof}
Let $\varphi$ be a smooth function on $\g^*$ with compact support.
Thus $\ll \Fcal_\g(f_k), \varphi\rr=\int_\g b(X)\tau(X/k) \widehat{\varphi}(X)dX$.

Let $N$ be a positive integer, and
$\tau_{\leq N}$ be the sum of the Taylor coefficients
$\tau_n$, up to degree $N$,
so $\tau=\tau_{\leq N}+\tau_{>N}$.
We write $I_k=\int_\g b(X)\tau(X/k) \widehat{\varphi}(X)dX$ as
$I_k^0+R_k$ with
$$
I_k^0=\int_\g b(X)\tau_{\leq N}(X/k) \widehat{\varphi}(X)dX=\sum_{k=0}^N \frac{1}{k^n}\int_\g b(X)\tau_{n}(X) \widehat{\varphi}(X)dX
$$
and $R_k= \int_\g b(X)\tau_{>N}(X/k) \widehat{\varphi}(X)dX$.

In multi-index notation, we  may write
(via an  integral formula depending of $\tau$ and its derivatives)
$\tau_{>N}(X)=\sum_{\alpha,|\alpha|=N+1} X^\alpha D_\alpha(X)$
with $D_\alpha(X)$  bounded by a polynomial function of $X$.
 If $\tau(X)$ depends smoothly of a parameter $x$, we can construct
 $D_\alpha(X)$ depending smoothly of $x$.
Thus $$R_k=\frac{1}{k^{N+1}}
 \sum_{\alpha,|\alpha|=N+1} \int_\g  b(X) X^\alpha D_\alpha(X/k)\widehat{\varphi}(X)dX.$$
 Since $\|X/k\|\leq \|X\|$, we can bound
 $b(X)X^\alpha D_\alpha(X/k)$ by a polynomial function of $X$.
 Since  $\widehat{\varphi}(X)$  is rapidly decreasing,
 we see that $|R_k|\leq \frac{c_N}{k^{N+1}}$.
\end{proof}


It is clear that if $b(X,k)=\sum_{m\in F} k^m b_m(X)$
is a sum over a finite set $F\subset \Z$
of smooth functions $b_m(X)$  with at most  polynomial growth,
 and $\tau(X)$ a function with polynomial growth as well as all its derivatives,
we can
obtain the asymptotic expansion of $\int_\g b(X,k)\tau(X/k)\widehat{\varphi}(X) dX$
by summing over $m\in F$  the asymptotic expansions of
$\int_\g b_m(X)\tau(X/k)\widehat{\varphi}(X) dX$
multiplied by $k^m$. So this is given by the Laurent series (in $1/k$)
 $$\sum_{m\in F}\sum_{n=0}^{\infty} k^m \frac{1}{k^n}
 \int_\g b_m(X) \tau_n(X)\widehat{\varphi}(X) dX.$$
We write this somewhat informally as

\begin{equation}\label{eq:informal}
\int_\g b(X,k)\tau(X/k)\widehat{\varphi}(X) dX\equiv
\int_\g b(X,k) \sum_{n=0}^{\infty}\tau_n(X/k)\widehat{\varphi}(X) dX.
\end{equation}
In short, we replace $\tau$ by its formal Taylor series $\sum_{n=0}^{\infty} \tau_n(X)$
and keep $b(X,k)$ as it is.

\subsection{Kirillov formula}
Recall that any coadjoint orbit $\Ocal\subset \g^*$ is provided with the Kirillov-Kostant-Souriau 
symplectic form $\Omega_\Ocal$ that is normalized as follows. We have
$\Omega_\Ocal\vert_\xi(X\cdot\xi,Y\cdot\xi)= \langle \xi, [X,Y]\rangle$, for $\xi\in\Ocal$ and $X,Y\in\ggot$.
Let $n_\Ocal=\frac{1}{2}\dim \Ocal$.
\begin{definition}
We denote by $\beta_\Ocal$ the distribution on $\g^*$ defined by the relation
$$
\ll\beta_\Ocal, \varphi\rr=\int_{\xi\in\Ocal}\varphi(\xi)\frac{(\Omega_\Ocal)^{n_\Ocal}}{(2\pi)^{n_\Ocal} n_\Ocal!}
$$
for any smooth function $\varphi$ on $\ggot^*$.
\end{definition}

We have chosen a system $\Delta^+\subset \t^*$ of positive roots,
and let $\rho\in \t^*$ be the corresponding element.
 We consider the positive Weyl chamber $\t^*_{\geq 0}$ with interior
 $\t^*_{> 0}$.
Identify $\wG$ to the discrete set  $A_G=(\rho+\Lambda)\cap \t^*_{> 0}$
of positive admissible regular elements of $\t^*$.

If $\lambda\in A_G$, we denote by $\chi_\lambda(g)$ the trace of the action of $g\in G$
 on the irreducible representation $V_\lambda$  parameterized by
$\lambda$.
We have the fundamental identity
\begin{equation}\label{eq:kirillov-V-lambda}
\chi_\lambda(e^X)j_{\g}^{1/2}(X) = \ll\beta_\lambda, e^{i\langle - ,X\rangle}\rr
\end{equation}
where $j_{\g}(X)=\det_{\g}\left(\frac{e^{X/2}-e^{-X/2}}{X}\right)$, and $\beta_\lambda:=\beta_{G\lambda}$.

\medskip

Let $\Rcal_\ggot: \Dcal'(\ggot^*)^G\longrightarrow \Dcal'(\tgot^*)^{W_G-alt}$ be the isomorphism introduced in Section \ref{sec:iso-R-g}.
If $\lambda\in A_G$, we have
\begin{equation}\label{eq:Har}
\Rcal_\ggot(\beta_\lambda)=\sum_{w\in W_G}\epsilon(w)\delta_{w\lambda}
\end{equation}
where $\delta_{w\lambda}$ is the $\delta$ function at $w\lambda$.

%

%
%

\subsection{Quasi-polynomials and asymptotics}
We recall the notion of piecewise quasi-polynomial functions  (for more details see \cite{pep-vergne:asymptotics}).

Let $\tilde{\Lambda}$ be the lattice generated by $\Lambda$ and $\rho$. We consider the vector space 
$E=\tgot^*\times \R$, which is equipped with the lattice $\tilde{\Lambda}\times \Z$.

A function $m:\tilde{\Lambda}\times\Z\to \C$ is periodic if there exists a positive integer D such that 
$m(x_0 + Dx) = m(x_0)$ for $x,x_0\in  \tilde{\Lambda}\times\Z$. By definition, the algebra of quasi-polynomial 
functions on $\tilde{\Lambda}\times\Z$ is generated by polynomials and periodic functions on $\tilde{\Lambda}\times\Z$.

To each closed rational polyhedron $P\subset \tgot^*$, we associate
\begin{itemize}
\item the sub-space $E_P\subset E$ generated by $(\xi,1),\xi\in P$,
\item the lattice $\tilde{\Lambda}_P:=(\tilde{\Lambda}\times \Z) \cap E_P$ in $E_P$,
\item the cone $C_P:=\{(t\xi,t), t\geq 0, \xi\in P\}\subset E_P$,
\item $[C_P]$, the characteristic function of $C_P$.
\end{itemize}

If $m_P$ is a quasi-polynomial function on the lattice $\tilde{\Lambda}_P$, we can form the product $m_P[C_P]$ that defines a function on 
$\tilde{\Lambda}\times\Z_{>0}$ as follows
$$
m_P[C_P](\lambda,k)=
\begin{cases}
   0\qquad\hspace{13mm} {\rm if}\ \lambda/k\notin P,\\
   m_P(\lambda,k)
   \qquad {\rm if}\  \lambda/k\in P.
\end{cases}
$$

\begin{definition}
A function $m:\tilde{\Lambda}\times \Z_{>0}\to \C$ is called a piecewise quasi-polynomial function if there is a collection
$\Acal$ of closed rational polyhedrons in $\tgot^*$ and a collection of quasi-polynomial functions\footnote{Each quasi-polynomial $m_P$ is defined on 
the sub-lattice $\tilde{\Lambda}_P\subset\tilde{\Lambda}\times \Z$. } $(m_P)_{P\in\Acal}$
such that
\begin{equation}\label{eq:piecewise-quasi-polynomial}
m=\sum_{P\in\Acal}  m_{P}[C_P].
\end{equation}
Here $\Acal$ is not necessarily finite but, in order that (\ref{eq:piecewise-quasi-polynomial}) makes sense, we assume that for any compact $\Kcal\subset\tgot^*$ the set $\{P\in\Acal, P\cap \Kcal\neq\emptyset\}$ is finite.

We denote by $\Scal(\tilde{\Lambda})$ the group of piecewise quasi-polynomial functions on $\tilde{\Lambda}\times \Z_{>0}$.
\end{definition}

To a piecewise quasi-polynomial function $m\in\Scal(\tilde{\Lambda})$, we associate the family of invariant distributions on $\g^*$
$$
\Theta_k(m)=k^{r}\sum_{\lambda\in A_G}m(\lambda,k)\beta_{\lambda/k},\quad k\geq 1.
$$

In a companion article \cite{pep-vergne:asymptotics}, we proved the following result.
\begin{prop}\label{prop:asymptotic-quasi-polynomial}
The family $\Theta_k(m)$ admits an asymptotic expansion
$$
\Theta_k(m)\equiv {\rm AS}_k(m):=\sum_{n\in \Z} \frac{1}{k^n}\theta_n(k).
$$
Here $(\theta_n(k))_{n\in\Z}$ is a collection of invariant distributions on $\g^*$, depending periodically of $k$, and such that $\theta_n(k)=0$ if $n< n_0^m$.
\end{prop}

For our computation, we need also a variation of Proposition \ref{prop:asymptotic-quasi-polynomial}.
Let $\rm j : \ggot\to \C$ be an invariant analytic function such that its Fourier transform
$B_{{\rm j}}:=\Fcal_\g({\rm j})$ is a compactly supported measure. We denote by 
$B_{\rm j}^k=\res_{\g^*}(k)B_{\rm j}$.

We consider now the family of distributions $B_{\rm j}^k\star\Theta_k(m)$ where $\star$ denotes the convolution.

\begin{prop}\label{prop:asymptotic-quasi-polynomial-star-B}
The family $B_{\rm j}^k\star\Theta_k(m)$ admits the asymptotic expansion
$$
B_{\rm j}^k\star\Theta_k(m)\equiv {\rm j}(i\partial/k){\rm AS}_k(m).
$$
\end{prop}
\begin{proof}Let $\varphi$ be a test function on $\g^*$. Let $R>0$ such that the ball $\{\xi\in\g^*, \|\xi\|\leq R/2\}$
contains the compact supports of $\varphi$ and $B_{\rm j}$.

We start with the relation
$$
\ll B_{\rm j}^k\star\Theta_k(m),\varphi \rr=\ll \Theta_k(m), \varphi_k\rr
$$
where
$$
\varphi_k(\xi)=\int_{\g^*}B_{\rm j}(\xi')\varphi(\xi+\xi'/k),\quad k\geq 1
$$
are smooth functions supported on the ball $\{\xi\in\g^*, \|\xi\|\leq R\}$. Let us write the Taylor series of $\varphi$ at $\xi$:
$\varphi(\xi+\eta)=\varphi(\xi)+\sum_{n=1}^N d_n\varphi(\xi)(\eta) + r_N(\xi,\eta)$. Here $\eta\mapsto d_n\varphi(\xi)(\eta)$ is an homogeneous polynomial of degree $n$ depending smoothly of the variable $\xi$, and $r_N$ is a smooth function
of $(\xi,\eta)$ such that
\begin{equation}\label{eq:maj-r-N}
|r_N(\xi,\eta)|\leq c_N \|\eta\|^{N+1},\quad \forall (\xi,\eta)\in \g^*\times\g^*,
\end{equation}
for some $c_N>0$.

If we write the Taylor series of ${\rm j}$ at $0$, ${\rm j}(X)=\sum_{n=0}^\infty{\rm j}_n(X)$, we see that
$$
\int_{\g^*}B_{\rm j}(\eta)d_n\varphi(\xi)(\eta)={\rm j}_n(-i\partial)\varphi(\xi).
$$
So the functions $\varphi_k$ admit the following description
$$
\varphi_k=\sum_{n=0}^N\frac{1}{k^n}{\rm j}_n(-i\partial)\varphi + R^N_k
$$
where $R^N_k(\xi)=\int_{\g^*}B_{\rm j}(\xi')r_N(\xi,\xi'/k)$. Thanks to (\ref{eq:maj-r-N}), we see that
there exists $C_N$ such that $|R^N_k(\xi)|\leq C_N/k^{N+1}$, $\forall\xi\in \g^*$. We check also that, for any $k\geq 1$, the support of $R_N^k$ is contained in the ball $\{\xi\in\g^*, \|\xi\|\leq R\}$.

Finally we obtain
\begin{equation}\label{eq=B-j-theta-k}
\ll B_{\rm j}^k\star\Theta_k(m),\varphi \rr=\sum_{n=0}^N\frac{1}{k^n}
\ll \Theta_k(m),{\rm j}_n(-i\partial)\varphi \rr + Rest^N_k
\end{equation}
where
$$
Rest^N_k=\ll \Theta_k(m),R^N_k \rr=k^r\sum_{\|\lambda\|/k\leq R}m(\lambda,k)\ll \beta_{\lambda/k},R_k^N\rr.
$$
The term $|Rest^N_k|$ is bounded by $\frac{C_N}{k^{N+1}}\sum_{\|\lambda\|/k\leq R}|m(\lambda,k)| \vol(G\lambda)$. 
Since $m(\lambda,k)$ is a piecewise quasi-polynomial function, we see that, for any $p\geq 1$, there exists $N\geq 1$ such that
\begin{equation}\label{eq=rest-N-k}
Rest^N_k=O(k^{-p}).
\end{equation}
Identities (\ref{eq=B-j-theta-k}) and (\ref{eq=rest-N-k}) show that the family
$B_{\rm j}^k\star\Theta_k(m)$ admits the asymptotic expansion ${\rm j}(i\partial/k){\rm AS}_k(m)$.
\end{proof}

\section{Spin quantization}

Let $M$  be an even dimensional oriented spin manifold of dimension $2d$ and provided with an action of $G$.
 If $(\CL,\nabla)$ is a $G$ equivariant line bundle equipped with a $G$ invariant Hermitian connection $\nabla$, we obtain a moment map $\phi_G:M\to \g^*$ and a closed two form $\Omega$ on $M$ using Kostant formula:
\begin{equation}\label{eq:kostant-rel}
\nabla^2=-i\Omega\quad \mathrm{and} \quad L(X)-\nabla_X=i\langle\phi_G,X\rangle.
\end{equation}
Here $X\in \g$ and  $L(X)$ is the infinitesimal action of $X\in \g$ on smooth sections of $\CL$.

Assume $M$  compact. The  spin quantization  $\QS_G(M,\CL)$ is a virtual finite dimensional representation of $G$, constructed as the index
of the Dirac operator on $M$ twisted by $\CL$. If $k$ is an integer, we denote by $\CL^k$ the $k$-th power of the line bundle $\CL$: we can
consider the equivariant index
 $$\QS_{G}(M,\CL^k)=\sum_{\lambda \in A_G} m_G(\lambda,k)V_{\lambda}.$$

Taking traces, we also write, for $g\in G$,
 $$\QS_{G}(M,\CL^k)(g)=\sum_{\lambda \in A_G} m_G(\lambda,k)\chi_{\lambda}(g).$$

We will need the following basic fact.

\begin{lemma}\label{eq:majoration-QS-L-k}
There exists a polynomial $P(k)$ such that $|\QS_{G}(M,\CL^k)(g)|\leq P(k)$ for all $g\in G$ and $k\geq 1$.
\end{lemma}

We associate to $(M,\Lcal)$  the distributions on $\g^*$ given by
\begin{equation}\label{eq:def-psi-k}
\Psi^{(M,\CL)}_k :=  \Fcal_\g \circ \res_\g(k) \Big(\QS_{G}(M,\CL^k)(e^{X})\Big)
\end{equation}
and
\begin{eqnarray}\label{eq:def-phi-k}
\Theta^{(M,\CL)}_k &:=&\Fcal_\g\circ \res_\g(k)\Big(\QS_{G}(M,\CL^k)(e^{X})j_\g^{1/2}(X)\Big)\nonumber\\
&=&\res_{\g^*}(k)\Big(\sum_{\lambda\in A_G} m_G(\lambda,k)\beta_\lambda\Big)\nonumber\\
&=& k^r\sum_{\lambda\in A_G} m_G(\lambda,k)\beta_{\lambda/k},
\end{eqnarray}
where $r=\dim (G/T)/2$. Thus $\Theta^{(M,\Lcal)}_k$ is the  re-scaled  geometric analogue of
$\QS_G(M,\CL^k)=\sum_{\lambda\in A_G} m_G(\lambda,k)\chi_{\lambda}$.

\subsection{Quasi-polynomial behavior of the multiplicities}\label{sec:quasi-polynomial-multiplicity}
We consider the multiplicity function $m_G: A_G\times\mathbb{Z}_{>0}\to \Z$
defined by the relation $\QS_{G}(M,\CL^k)=\sum_{\lambda \in A_G} m_G(\lambda,k)\chi_{\lambda}$. We extend $m_G$ to 
$\tilde{\Lambda}\times \Z_{>0}$ by defining $m_G(\lambda,k)=0$ if $\lambda\notin A_G$.

The main objective of this section is the following result.

\begin{theo}\label{theo:m-G-piecewise}
Suppose that the $G$-action on $M$ admits {\em abelian infinitesimal stabilizers}. Then $m_G\in \Scal(\tilde{\Lambda})$.
\end{theo}

In particular, Theorem \ref{theo:m-G-piecewise} implies that the function $k\mapsto m_G(k\lambda,k)$ is a quasi-polynomial function of $k\geq 1$.

Theorem \ref{theo:m-G-piecewise} is a consequence of the geometric formulas for the multiplicities $m_G(\lambda,k)$ obtained in 
\cite{pep-vergne:acta}. Let us first recall these geometric formulas.

Let $\tilde{T}\to T$ be the covering such that $\tilde{\Lambda}$ is the weight lattice of the torus $\tilde{T}$. 
Every element $\lambda\in \tilde{\Lambda}$ determines a character of $\tilde{T}$, and we denote by $\C_\lambda$ 
the corresponding $1$-dimensional representation.

Recall that we have chosen  a  $G$-invariant Hermitian connection on $\CL$, and $\phi_G: M\to \g^*$
is the associated moment map. We consider the positive Weyl chamber $\t^*_{\geq 0}$ with interior $\t^*_{> 0}$ 
and the following subset
$$
Y=\phi_G^{-1}(\t^*_{> 0}).
$$

Then $Y$ is a $T$-invariant sub-manifold of $M$, not necessarily connected, but
every connected component of $Y$ is even dimensional. We denote by $\phi_T: Y\to \t^*$ the restriction of the map $\phi_G$ to $Y$. Thus $\phi_T(Y)$ is contained in $\t^*_{>0}$.

\begin{lemma}
 The $K$-equivariant spin structure on $M$ induces a $\tilde{T}$-equivariant spin structure on $Y$. We denote by $\Scal_Y$  the corresponding spinor bundle.
\end{lemma}

Now we explain how we can attach a multiplicity function
$$
m_j: \tilde{\Lambda}\times\mathbb{Z}_{>0}\to \mathbb{Z}
$$
to a connected component $Y_j$ of $Y$. Let $\t_j\subset\t$ be the generic infinitesimal stabilizer of the $T$-action on $Y_j$.
We denote by $I_j\subset\t^*$ the affine rational subspace with direction $\t_j^\perp$ that contains $\phi_T(Y_j)$.

We consider a finite collection $\Bcal_j$ of affine co-dimension $1$ subspaces of $I_j$ defined as follows: 
$A\in\Bcal_j$ if $A=\phi_T(y)+\t_y^\perp$ for some
$y\in Y_j$ such that $\dim (\t_y\slash \t_j)=1$.

\begin{definition} A chamber   of $I_j$ is a connected
component of the open subset $(I_j)_{reg}:=(I_j\setminus \cup_{E\in \Bcal_j} E)\cap \t^*_{>0}$.
\end{definition}

Thus, the closure of a chamber is a rational polyhedron contained in $\t^*_{\geq 0}$.

Let $\cgot$ be a chamber of $I_j$ and let  $\xi\in\cgot$ be  a regular value of the map $\phi_T:Y_j\to I_j$. We consider the orbifold
$$
\Ycal_{j,\xi}:=(\phi_T^{-1}(\xi)\cap Y_j)/T.
$$

\begin{lemma}Let $(\lambda,k)\in \tilde{\Lambda}\times\mathbb{Z}_{>0}$ such that $\frac{\lambda}{k}\in I_j$.
The $\tilde{T}$-equivariant spinc bundle $\mathcal{S}_{Y_j}\otimes \CL^k\vert_{Y_j}\otimes \C_{-\lambda}$ on $Y_j$ induces a spinc bundle
$\Scal_{j,\xi}^{\lambda, k}$ on $\Ycal_{j,\xi}$. We denote by $\Qcal(\Ycal_{j,\xi}, \Scal_{j,\xi}^{\lambda, k})$ the index of the
corresponding Dirac operator.
\end{lemma}

Let $E_j\subset \tgot^*\times \R$ be the subspace generated by $(v,1), v\in I_j$. We denote by 
$\tilde{\Lambda}_j:=(\tilde{\Lambda}\times \Z) \cap E_j$ the corresponding lattice of $E_j$. 
The Kawasaki index theorem \cite{Kawasaki81} tells us that the map
$$
(\lambda,k)\longmapsto \Qcal(\Ycal_{j,\xi}, \Scal_{j,\xi}^{\lambda, k})
$$
is a quasi-polynomial  function ${\rm m}_{j,\xi}$ on $\tilde{\Lambda}_j$. 

The following result is proved in \cite{pep-vergne:acta}.

\begin{theorem}
$\bullet$ The quantity $\Qcal(\Ycal_{j,\xi}, \Scal_{j,\xi}^{\lambda, k})$ does not depend on
the choice of $\xi\in \cgot$. We denote by ${\rm m}_{j,\cgot}:\tilde{\Lambda}_j\to \C$ the corresponding quasi-polynomial.

$\bullet$ Let $(\lambda,k)\in \tilde{\Lambda}\times\mathbb{Z}_{>0}$ and let $\cgot_1,\cgot_2$ be two chambers such that 
$\frac{\lambda}{k}\in \overline{\cgot_1}\cap\overline{\cgot_2}$. Then
$$
{\rm m}_{j,\cgot_1}(\lambda, k)={\rm m}_{j,\cgot_2}(\lambda, k).
$$
\end{theorem}

The first point of the previous theorem tell us that ${\rm m}_{j,\cgot}$ is the zero map if the chamber $\cgot$ is not
contained in the image of the map $\Phi_T:Y_j\to I_j$. It is due to the fact that we can choose a
regular value $\xi\in\cgot$ with empty reduced space  $\Ycal_{j,\xi}$.

Thus we can define multiplicity functions $(m_j)_{j\in J}$ parameterized by the set $J$ of connected components of the manifold $Y$ as follows.

\begin{definition}The map $m_j: \tilde{\Lambda}\times\mathbb{Z}_{>0}\to \mathbb{Z}$ is defined as follows
$$
m_j(\lambda,k)=
\begin{cases}
   0\qquad\hspace{22mm} {\rm if}\ \frac{\lambda}{k}\notin I_j,\\
   {\rm m}_{j,\cgot}(\lambda, k)
   \qquad \hspace{7mm}{\rm if}\ \frac{\lambda}{k}\in \overline{\cgot}\subset I_j.
\end{cases}
$$
\end{definition}

We can now state the main result of \cite{pep-vergne:acta}.
\begin{theorem}\label{theo:QR-multiplicities}
The following relation
$$
m_G(\lambda,k)=\sum_{j\in J}m_j(\lambda,k)
$$
holds for any $(\lambda,k)\in \tilde{\Lambda}\times\mathbb{Z}_{>0}$.
\end{theorem}

\medskip

Let us fix $j$. Let $\Acal_j$ be the collection of closed rational polyhedrons formed by the faces of the closures $\overline{\cgot}$ 
of the chambers in $I_j$ (they are contained in $\tgot^*_{\geq 0}$).
To each $P\in \Acal_j$, we associate
\begin{itemize}
\item the sub-space $E_P\subset \t^*\times\R$ generated by $(v,1),v\in P$ and the lattice $\tilde{\Lambda}_P:=(\tilde{\Lambda}\times \Z) \cap E_P$,
\item the cone $C_P:=\{(t\xi,t), t\geq 0, \xi\in P\}\subset E_P$,
\item the quasi-polynomial $m_P:\tilde{\Lambda}_P\to \C$ such that $m_P(\lambda,k)={\rm m}_{j,\cgot}(\lambda,k)$, for 
$(\lambda,k)\in \tilde{\Lambda}_P$,  if $P$ is a face of 
$\overline{\cgot}$. 
\end{itemize}

By inclusion-exclusion, we see that the multiplicity function $m_j$ admits the decomposition
\begin{equation}\label{eq:m-G-decomposition}
m_j=\sum_{P\in\Acal_j} \alpha_P m_{P}[C_P],
\end{equation}
for some appropriate choice of constants $\alpha_P\in \Z$. In other words, the multiplicity function $m_j$ is in the space $\Scal(\tilde{\Lambda})$. 
The relation $m_G=\sum_j m_j$ given by Theorem \ref{theo:QR-multiplicities} shows that $m_G\in \Scal(\tilde{\Lambda})$. The proof of Theorem \ref{theo:m-G-piecewise} is completed. $\Box$

\medskip

\begin{example}
Let us give a simple example for the group $G=SU(2)$.
Consider  $\t^*$ with basis $\rho$. In this basis $A_G=\{\lambda; \lambda\in\Z, \lambda>0\}$.
We consider the line bundle $\CL$ associated to $(\rho,\rho)$ on the spin manifold $M=G/T\times G/T$. We see that $\QS_G(M,\CL)$ is the tensor product
$V(k\rho)\otimes V(k\rho)$, and $V(k\rho)$ is the irreducible representation of dimension $k$ of $G$. So $\QS_G(M,\CL)$ is the sum of the representations $V(j\rho)$, with $j$ odd and less than $2k$.
If we write $\QS_G(M,\CL^k)=\sum_{\lambda>0} m_G(\lambda,k)V(\lambda \rho)$,
 $$m_G(\lambda,k)=\frac{1}{2}(1-(-1)^\lambda), \ {\rm for}\ 1\leq \lambda\leq 2k.$$
\end{example}

%
%
%
%
%
%
%
%

\subsection{Equivariant cohomology}

Let $N$ be a $G$-manifold and let $\A(N)$ be the space of differential forms on $N$,
graded by its exterior degree.
Following \cite{ber-ver82} and \cite{Witten82},
an equivariant form is a $G$-invariant smooth function
$\alpha: \g\to \A(N),$ thus
$\alpha(X)$ is a differential form on $N$ depending smoothly of $X\in \g$.
Consider the operator
\begin{equation}\label{DX}
d_\g\alpha(X)=d\alpha(X)-\iota(v_X)\alpha(X)
\end{equation}
 where $\iota(v_X)$ is the contraction by the vector field $v_X$
generated by the action of $-X$ on $N$.
Then $d_\g$ is an odd operator with square $0$,
and  the equivariant cohomology is defined
to be the cohomology space of $d_\g$.
It is important to note that the dependence of $\alpha$ on $X$ may be $C^\infty$.
If the dependence of $\alpha$ in $X$ is polynomial, we denote by
$H^*_G(N)$ the corresponding  $\Z$-graded algebra.
By definition, the grading of $P(X)\otimes \mu$, $P$ a homogeneous polynomial and $\mu$ a  differential form on $N$,
is the exterior degree of $\mu$ plus twice the polynomial degree in $X$.

The line bundle  $\CL\to M$ determines the closed equivariant form
$\Omega(X)=\Omega-\ll \phi_G,X\rr $ on $M$.
Here we will not assume any non degeneracy condition on $\Omega$.

Choose a  $G$-invariant Riemannian metric on $M$.
Let $$J(A)=\det{}_{\R^{2d}}\left( \frac{e^{A/2}-e^{-A/2}}{A}\right),$$ an invariant function of $A$.
 Then $J(0)=1$. Consider $\frac{1}{J^{1/2}(A)}$ and its Taylor expansion at $0$:
 $$\frac{1}{J^{1/2}(A)}=\det{}_{\R^{2d}}^{1/2}(\frac{A}{e^{A/2}-e^{-A/2}})= \sum_{n=0}^{\infty}  B_{n}(A).$$

 Each  function $B_{n}(A)$ is an invariant polynomial of degree $n$ (in fact $B_n$ is non zero only for $n$ even)  and by the Chern Weil construction, and choice of connections on $TM$,  the function
  $B_{n}$
determines a closed equivariant form
${\Ahat}_{n}(M)(X)$ on $M$ of homogeneous equivariant degree $2n$. Remark that ${\Ahat}_0(M)=1$.
We define the formal series of equivariant  forms:
   $${\Ahat}(M)(X)=\sum_{n=0}^{\infty} {\Ahat}_{n}(M)(X).$$

As $M$ is compact, we can find a positive constant $r_M$
such that, for $\|X\|<r_M$, these series of equivariant forms  is convergent.
   In particular ${\Ahat}(M)(0)$ is a closed differential form on $M$ which represents the usual $\Ahat$ class of $M$.

\subsection{The equivariant index}

Recall  the ``delocalized index formula" (see \cite{BGV}). For $X\in \g$ such that $\|X\|<r_M$,
so that $\Ahat(M)(X)$ is well defined, we have
\begin{equation}\label{eq:BV}
\QS_G(M,\CL)(\exp X)=\frac{1}{(-2i\pi)^{d}}\int_{M} e^{-i\Omega(X)}\Ahat(M)(X).
\end{equation}

Here  $d=\dim M/2$. Note that Kostant relations (\ref{eq:kostant-rel}) implies that the equivariant form
$\Omega(X)=\Omega-\langle\phi_G,X\rangle$ is closed. In other words $\iota(X)\Omega+ d\langle\phi_G,X\rangle=0$
for any $X\in\g$.

For each integer $n$, consider
the analytic function on $\g$ given by
$$\Ical_n(X)=\frac{1}{(-2i\pi)^{d}}\int_{M} e^{-i \Omega(X)}{\Ahat}_n(M)(X).$$

There is a remarkable relation between the character associated to  $\CL^k$ and the dilation  $X\to  X/k$   on $\g$.
\begin{lemma}\label{crucial}
When $X\in \g$  is such that $\|X\|<r_M$, then, for any $k\geq 1$, one has
$$\QS_G(M,\CL^k)(\exp (X/k))=k^d\sum_{n=0}^{\infty} \frac{1}{k^n} \Ical_n(X).$$
\end{lemma}

\begin{proof}
Write $\Ahat(M)(X)=\sum_{n=0}^{\infty} \tau_{n}(X)$
as a sum of forms with  coefficients homogeneous polynomials in $X$ of degree $n$.
Thus $$\Ahat_n(M)(X)=\sum_{q\leq n}(\tau_q(X))_{[n-2q]}$$ where
 $\omega_{[s]}$ is the component of exterior degree $s$ of a differential form
 $\omega$.

   For $\|X\|<r_M$, $\sum_{n=0}^{\infty}\tau_{n}(X)$ is a
convergent series with sum the equivariant $\Ahat$ class.
We obtain
$$
\QS_G(M,\CL^k)(\exp (X/k))=
\frac{1}{(-2i\pi)^{d}}\int_{M} e^{-i k\Omega(X/k)} (\sum_{n=0}^{\infty}\tau_{n}(X/k)).
$$

Let $b(X,k)=e^{-i k\Omega(X/k)}= e^{-i k\Omega} e^{i\ll\phi_G,X\rr}$.
Remark that $b(X,k)$ depends polynomially of $k$.
We rewrite
\begin{equation}\label{eq:gradedstupid}
\QS_G(M,\CL^k)(\exp (X/k))=
\frac{1}{(-2i\pi)^{d}}\int_{M} b(X,k) (\sum_{n=0}^{\infty}\frac{1}{k^n}\tau_{n}(X)).
\end{equation}

Only the  top exterior degree term  contributes to the integral. We compute it in terms of the equivariant forms $\Ahat_n(M)$.
\begin{lemma}\label{lem:gradedstupidcompact}
$$
\Big[b(X,k) \Big(\sum_{n=0}^{\infty}\frac{1}{k^n}\tau_{n}(X)\Big)\Big]_{[2d]}
=k^d \Big[ e^{-i\Omega(X)}\Big(\sum_{n=0}^{\infty}\frac{1}{k^n}\Ahat_{n}(M)(X)\Big)\Big]_{[2d]}.
$$
\end{lemma}
\begin{proof}
The coefficient of $k^d \frac{1}{k^n}$ in the left hand side is
$$\sum_{a=0}^d \frac{(-i\Omega)^a}{a!} (\tau_{n+a-d}(X))_{[2d-2a]} e^{i\ll\phi_G,X\rr}.$$
This is the term of exterior degree $2d$ of
$e^{-i\Omega(X)}\Ahat_n(M)(X)$.
\end{proof}

This identity implies the lemma.
\end{proof}

\subsection{Twisted Duistermaat-Heckman distributions}

Let $\nu$ be a closed equivariant form on $M$ with polynomial coefficients.
Let $-i\Omega(X)$ $=-i\Omega+i\ll\phi_G,X\rr $ be our equivariant curvature.
As $\nu(X)$ depends polynomially on $X$,
$\int_M e^{-i\Omega(X)}\nu(X)$ is a function on $\g$ of at most polynomial growth.
\begin{definition}
We denote by $DH^G(M,\Omega,\nu)\in\Dcal'(\g^*)^G$ the Fourier transform of
$\frac{1}{(-2i\pi)^{d}}\int_M e^{-i\Omega(X)}\nu(X)$.
\end{definition}

We thus have, for $\varphi$ smooth with compact support,
\begin{eqnarray*}
\ll DH^G(M,\Omega,\nu),\varphi\rr &=&\frac{1}{(-2i\pi)^{d}}\int_{M}\int_{\g} e^{-i\Omega(X)}\nu(X)\widehat{\varphi}(X)dX\\
&=&\frac{1}{(-2i\pi)^{d}}\int_{M} e^{-i\Omega}[\nu(-i\partial)\varphi](\phi_G(m)).
\end{eqnarray*}

If $\nu(X)=\sum_a p_a(X)\nu_a$ with polynomial functions $p_a$ on $\g$ and differential forms $\nu_a$ on $M$, the expression
$$[\nu(-i\partial)\varphi](\phi_G(m))= \sum_a [p_a(-i\partial)(\varphi)](\phi_G(m)) \nu_a$$ is a differential form on $M$ supported on the pull-back by
$\phi_G$ of the  support of $\varphi$. This shows, and this will be used in our application to a line bundle with
proper moment map,
that $DH^G(M,\Omega,\nu)$ is a well defined distribution supported on $\phi_G(M)$, provided $\phi_G$ is proper. Notice that
the distribution $DH^G(M,\Omega,\nu)$ depends only of the class $[\nu]\in H_G^*(M)$.

Recall that we have associated to $(M,\Lcal)$ the distributions on $\g^*$ given by
\begin{eqnarray*}
\Psi^{(M,\Lcal)}_k &:=&  \Fcal_\g \circ \res_\g(k) \Big(\QS_{G}(M,\CL^k)(e^{X})\Big),\\
\Theta^{(M,\Lcal)}_k &:=& \Fcal_\g\circ \res_\g(k)\Big(\QS_{G}(M,\CL^k)(e^{X})j_\g^{1/2}(X)\Big).
\end{eqnarray*}

The distribution $\Theta^{(M,\Lcal)}_k$ is the weighted sum $k^r\sum_{\lambda\in A_G} m_G(\lambda,k)\beta_{\lambda/k}$
of the canonical measures on the orbits $G\lambda/k$.

\begin{theorem}\label{asympt1}
Let $M$ be a compact spin even dimensional oriented manifold with a $G$ equivariant line bundle $\CL$.
When $k$ tends to $\infty$, we have the asymptotic expansions
\begin{equation}\label{eq:asymptotics-psi-k}
\Psi^{(M,\Lcal)}_k\equiv k^{d}\  \sum_{n=0}^{\infty}
   \frac{1}{k^n} DH^G(M,\Omega, \Ahat_n(M))
\end{equation}
and
\begin{equation}\label{eq:asymptotics-theta-k}
\Theta^{(M,\Lcal)}_k \equiv \  j_\g^{1/2}(i\partial/k)\Big(k^d\sum_{n=0}^{\infty}
   \frac{1}{k^n} DH^G(M,\Omega, \Ahat_n(M))\Big) .
\end{equation}

\end{theorem}

\begin{proof}
 We present  a detailed proof since we will have to
adapt this proof to the case of a non compact manifold $M$.

We fix a $G$-invariant function $h:\g\to\R$ equal to $1$    on a  small ball $\|X\|\leq r$ and
with compact support contained in $\|X\|<r_M$,  so that $h(X){\hat A}(M)(X)$  is a smooth  compactly supported function on
$\g$ with value differential forms on $M$.

Let $\varphi$ be a  smooth compactly supported function on $\g^*$.
We have $\ll\Psi^{(M,\Lcal)}_k,\varphi\rr= \mathrm{I}_k+ \mathrm{J}_k$ where
$$
\mathrm{I}_k=
\int_{\g}\QS_G(M,\Lcal^k)(e^{X/k})h(X/k)\widehat{\varphi}(X)dX,
$$
and
$$
\mathrm{J}_k=\int_{\g}\QS_G(M,\Lcal^k)(e^{X/k})(1-h(X/k))\widehat{\varphi}(X)dX.
$$
 We will see that $\mathrm{J}_k=O(k^{-\infty})$.
For estimating $\mathrm{I}_k$,
we will use  the delocalized index formula (\ref{eq:BV}).

  Using the fact that  $|\QS_{G}(M,\CL^k)(g)|$ is uniformly bounded by a polynomial in $k$,  that
  the support of $(1-h(X/k))$  is contained in the set $\|X\|\geq r k$, $r>0$, and that the function
  $\widehat{\varphi}$ is rapidly decreasing, we see that $\mathrm{J}_k=O(k^{-\infty})$.

We now analyze $\mathrm{I}_k$.
We can use Formula (\ref{eq:BV}) since $\|X/k\|\leq r_M$.
Thus
$$
\mathrm{I}_k=
\frac{1}{(-2i\pi)^{d}}\int_{M} \int_\g e^{-ik\Omega(X/k)}\Ahat(M)(X/k) h(X/k)\widehat{\varphi}(X)dX.
$$

Proposition \ref{prop:technical-asymptotics} extends to the case of vector valued
functions on $\g$.
Let $E$ be a finite dimensional space,
and let $\tau:\g\to  \Lambda E$ be a smooth function  with at most polynomial growth,
 as well as all its derivatives.
The coefficients $\tau_n(X)$ of its Taylor series $\sum_{n=0}^{\infty} \tau_n(X)$
are   $\Lambda E$ valued polynomial functions  on $\g$ homogeneous of degree $n$.
Let $b(X,k)=\sum_{m=0}^d k^m b_m(X)$
where $b_m(X)$  are   smooth functions  of $X$ with value in $\Lambda E$ with at most polynomial growth.
 Thus
$$b(X,k)\sum_{n=0}^{\infty}\tau_n(X/k)=\Big(\sum_{m=0}^d k^m b_m(X)\Big)
\Big(\sum_{n=0}^{\infty} \frac{1}{k^n}\tau_n(X)\Big)$$
is a  Laurent series in $\frac{1}{k}$
 of functions on $\g$ with values in $\Lambda E$.

We have
\begin{equation}\label{eq:asympbtau}
\int_\g b(X,k) \tau(X/k) \widehat{\varphi}(X) dX\equiv
\int_\g b(X,k) \Big(\sum_{n=0}^{\infty} \tau_n(X/k)\Big)
\widehat{\varphi}(X) dX.
\end{equation}

In short, we replace $\tau$ by its Taylor series.

\bigskip

Consider the differential form
  $h(X){\hat A}(M)(X)$.
  For each $x\in M$, $\tau^h_x(X)=h(X) {\hat A}(M)|_x(X)$
  is a smooth compactly supported function on $\g$ with values in $\Lambda T_x^*M$.
  Let $$b_x(X,k)=
  e^{i\langle\phi_G(x),X\rangle} e^{-ik\Omega_x}.$$
  It depends polynomially of $k$ and each coefficient of this polynomial is a bounded function of $X$ (as it is proportional to $e^{i\ll\phi_G(x),X\rr}$).
  So we can certainly apply the asymptotic formula (\ref{eq:asympbtau}).

  Let $\mathrm{I}_k^x=\frac{1}{(-2i\pi)^d}\int_\g b_x(X,k)\tau^h_x(X/k) \widehat{\varphi}(X) dX$,
  with values in $\Lambda T^*_xM$. So $x\mapsto \mathrm{I}_k^x$ is a differential form
  and  $\mathrm{I}_k=\int_{x\in M}\mathrm{I}_k^x.$

The manifold $M$ is compact.
 Proposition \ref{prop:technical-asymptotics} gives  us an asymptotic expansion for $\mathrm{I}_k^x$, and
 the proof shows that the rest is bounded uniformly.
 So the asymptotic expansion  of $\mathrm{I}_k$ is obtained by replacing (at each $x\in M$)
 $\tau^h_x(X/k)$ by its Taylor series.
 As $h$ is identically $1$ near zero, we obtain in the notations of the proof of Lemma \ref{crucial},
$$
 \mathrm{I}_k\equiv  \frac{1}{(-2i\pi)^d}\int_M \int_\g b(X,k)
 \Big(\sum_{n=0}^{\infty} \frac{1}{k^n}\tau_n(X)\Big)\widehat{\varphi}(X) dX.
$$
By Lemma \ref{lem:gradedstupidcompact}, this is
$$k^d \frac{1}{(-2i\pi)^d}\int_M \int_\g e^{-i\Omega(X)}
\Big(\sum_{n=0}^{\infty} \frac{1}{k^n}\Ahat_n(M)(X)\Big)\widehat{\varphi}(X) dX.$$

 By definition of the twisted Duistermaat-Heckman measure, we obtain
  the expansion of Formula (\ref{eq:asymptotics-psi-k}).

 The proof of Expansion formula (\ref{eq:asymptotics-theta-k}) is identical.
\end{proof}

\subsection{Asymptotics localized at $g\in G$}

For applications to formal geometric quantization, we state an analogous asymptotic descent formula.
Now let $g\in G$.  Let $G(g)$ be the centralizer of $g$ in $G$, and
 $\g(g)$ its Lie algebra.

Consider the manifold $M^g$ of fixed points of the action of $g$ on $M$: it may have several connected components $M^g_a$, all of even dimension. Since $M$ is spin, $M^g$ admits an orientation (that is all connected components are oriented). Let $\Ncal^g$ be the normal bundle of
$M^g$ in $M$ and let
$$
R_{\Ncal^g}(X)=R_{\Ncal^g}+\mu_{\Ncal^g}(X)
$$
be its equivariant curvature \cite{BGV}.

We denote by $\Omega_g(X)$ the restriction of $\Omega(X)$ on $M^g$. We consider the following equivariant form
$$
D_g(M)(X)=\det{}^{1/2}_{\Ncal^g}\Big(I-(g^{-1})^{\Ncal^g}e^{-R_{\Ncal^g}(X)}\Big).
$$
Here the square root is chosen such that $\det_{\Ncal^g}^{1/2}(I-(g^{-1})^{\Ncal^g})>0$.

For $X\in\g(g)$ small enough, we have
\begin{equation}\label{eq:QS-g-exp-X}
\QS_G(M,\Lcal)(g\exp(X))=\int_{M^g} c_g \frac{\widehat{A}(M^g)(X)}{D_g(M)(X)}
\, g^{\Lcal}\, e^{-i\Omega_g(X)}
\end{equation}
where

$\bullet$ $g^{\Lcal}$ is the locally constant function
on $M^g$ with value a complex number of modulus $1$ given by the action of $g$ on the fiber of $\CL\vert_{M^g}$.
We write $u_a$ for the restriction of $g^{\Lcal}$ to a connected component $M^g_a$. This is a complex number of
modulus 1. If $g$ is of finite order, $u_a$ is a root of unity.

$\bullet$ $c_g=\epsilon_g(-2i\pi)^{-\dim M/2}(2\pi)^{\mathrm{rk}\, \Ncal^g/2}$ is a locally constant function
independent of $\CL$. Here the sign $\epsilon_g=\pm$ depends only of the action of $g$ on the spinor bundle $\Scal\vert_{M^g}$.


Consider the equivariant form
$V_g(\CL,k)(X)=c_g (g^{\Lcal})^k \frac{\widehat{A}(M^g)(X)}{D_g(M)(X)}$ on $M^g$.
Then we have $V_g(\CL,k)(X)=\sum_{n=0}^{\infty} V_g(n,k)(X)$
where $V_g(n,k)(X)$ are $G(g)$-equivariant forms on $M^g$ homogeneous of even degree $2n$.

Thus, for $X\in \g(g)$ small,
we have
$$
\QS_G(M,\Lcal^k)(g\exp(X))= \sum_{n=0}^{\infty}\int_{M^g}e^{-ik\Omega_g(X)} V_g(n,k)(X).
$$

Here Lemma \ref{crucial} becomes
\begin{lemma}\label{cruc-g}
When $X\in \g(g)$ is sufficiently small, then, for any $k\geq 1$, one has\footnote{The term $k^{\frac{\dim M^g}{2}} \Ical_g(n,k)(X)$ must be understood as the sum $\sum_a k^{\frac{\dim M^g_a}{2}}\int_{M_a^g} e^{-i\Omega_g(X)} V_g(n,k)(X)$.}
$$\QS_G(M,\CL^k)(g\exp (X/k))=\sum_{n=0}^{\infty} \frac{1}{k^{n}}k^{\frac{\dim M^g}{2}} \Ical_g(n,k)(X)$$
where $\Ical_g(n,k)(X):=\int_{M^g} e^{-i\Omega_g(X)} V_g(n,k)(X)$.
\end{lemma}

Now consider the character $\chi_\lambda$ of the irreducible representation attached to the
admissible orbit $\Ocal_\lambda$. If $g\in G$, the fixed point set $\Ocal_\lambda^g$ is a union of
coadjoint orbits  of the connected component
$G(g)^o$ of $G(g)$. They are all of the same even dimension $r_g=\dim G(g) -\dim T$.

Then we have the formula, for $X\in \g(g)$,
\begin{equation}\label{eq:kirillov-g}
\chi_\lambda(g e^{X})j_{\g(g)}^{1/2}(X)\det{}_{\g/\g(g)}^{1/2}(1-g^{-1}e^{-X})=
 \ll\beta(g,\lambda)(\xi),e^{i \ll \xi,X\rr }\rr.
 \end{equation}
Here the square root is determined such that
 $\det_{\g/\g(g)}^{1/2}(1-g^{-1})>0$ and $\beta(g,\lambda)$ is a measure on $\CO_\lambda^g$ entirely determined by this equation.

Let $\tilde T$ be the torus with weight lattice $\tilde \Lambda$ (the lattice generated by $\Lambda$ and $\rho$). Then
 any $\lambda\in A_G$ defines a character $\tilde{t}\in \tilde{T}\to \tilde{t}^\lambda$. 
Suppose that $g\in T$ and denote by $W_{G(g)}\subset W_G$ the Weyl group of $G(g)^o$. Then
$$
\Ocal_\lambda^g=\bigcup_{\overline{w}\in W_{G(g)}\backslash W_G} G(g)^ow\lambda.
$$

Let $\Rcal_{\g(g)}:\Dcal'(\g(g))^{G(g)^o}\to \Dcal'(\t)^{W_{G(g)}-alt}$ be the canonical isomorphism.

%

 \begin{lemma} \label{lem:Harbeta}
 Choose $\tilde{g}\in \tilde{T}$ above $g\in T$.

 $\bullet$ The measure $\beta(g,\lambda)$ satisfies the relation 
 $$
 \beta(g,\lambda)=\gamma_{\tilde{g}}\sum_{\overline{w}\in W_{G(g)}\backslash W_G} \tilde{g}^{w\lambda}\beta_{G(g)^o w\lambda}
 $$
 where $\gamma_{\tilde{g}}$ is a complex number of modulus $1$.

$\bullet$ We have
$\Rcal_{\g(g)}\left(\beta(g,\lambda)\right)=\gamma_{\tilde{g}} \sum_{w\in W_G} \epsilon(w) \tilde{g}^{w\lambda}\delta_{w\lambda}$.

 \end{lemma}

 \begin{proof}
The first point follows from the character formula and the second is a direct consequence of the first.
\end{proof}

\begin{definition}
We denote by $DH^{G(g)}(M^g,\Omega_g,V_g(n,k))\in\Dcal'(\g(g)^*)^{G(g)}$ the Fourier transform of the function
$\Ical_g(n,k)$ (see Lemma \ref{cruc-g}).
\end{definition}

Then for a test function $\varphi$,
$\ll DH^{G(g)}(M^g,\Omega_g, V_g(n,k)),\varphi\rr$ is of the form
$\sum_a u_{a}^k \ll s_{a,n},\varphi\rr$ with $\ll s_{a,n},\varphi\rr=
\int_{M^g_a} \int_{\g(g)}e^{-i\Omega_g(Y)}\alpha_{a,n}(Y)\widehat{\varphi}(Y)dY$ : here $\alpha_{a,n}(Y)$ is an equivariant form
on $M^g _a$ of degree $2n$, and $u_a$ is the restriction of $g^{\Lcal}$ to $M^g_a$.

Asymptotic expansion of distributions of the form
$\sum_a \sum_{n=0}^{\infty} u_a^k k^{-n}h_{n,a}$ are thus well defined and
the coefficients $H(n,k)=\sum_a u_a^kh_{n,a}$ are uniquely determined.

With the same proof than Theorem \ref{asympt1}, we obtain the following theorem.
\begin{theorem}
Let $M$ be a compact spin even dimensional oriented manifold with a $G$ equivariant line bundle $\CL$.
Let $g\in G$ of finite order. When $k$ tends to $\infty$, we have the asymptotic
equivalence\footnote{The term $k^{\frac{\dim M^g}{2}}DH^{G(g)}(M^g,\Omega_g, V_g(n,k))$ must be understood as the sum
$\sum_a k^{\frac{\dim M^g_a}{2}}DH^{G(g)}(M^g_a,\Omega_g, V_g(n,k))$.}
$$\res_{\g(g)^*}(k)\Big(\sum_{\lambda\in \Lambda} m(\lambda,k) \beta(g,\lambda)\Big)\equiv$$
 $$ j_{\g(g)}^{\frac{1}{2}}(i\partial/k)\det{}_{\g/\g(g)}^{\frac{1}{2}}(1-g^{-1} e^{-i\partial/k})
 \sum_{n=0}^{\infty} \frac{1}{k^{n}}  k^{\frac{\dim M^g}{2}}DH^{G(g)}(M^g,\Omega_g, V_g(n,k)).
 $$
\end{theorem}

\section{Spin quantization in the non-compact setting}

In this section we work with an even dimensional oriented spin manifold $M$ of dimension $2d$ provided with an action of $G$.
We do not assume that $M$ is compact. Let $(\CL,\nabla)$ be
a $G$ equivariant line bundle equipped with a $G$-invariant Hermitian connection
$\nabla$.

We assume that the moment map $\phi_G:M\to \g^*$ defined by the Kostant relation (\ref{eq:kostant-rel}) is a proper map.
In the next section
we explain how is defined the formal geometric quantization of the data $(M,\phi_G,\Lcal)$.

Notice that the twisted Duistermaat-Heckman distributions \break $DH^G(M,\Omega,\nu)\in\Dcal'(\g^*)^G$ are
still defined by the relation
\begin{equation}\label{eq:DH-noncompact}
\ll DH^G(M,\Omega,\nu),\varphi\rr =\frac{1}{(-2i\pi)^{d}}\int_{M} e^{-i\Omega}[\nu(-i\partial)\varphi](\phi_G(m)).
\end{equation}
Here $\nu$ is a closed equivariant form on $M$ with polynomial coefficients and $\varphi$ is a test function on $\g^*$.
The integral in (\ref{eq:DH-noncompact}) is well-defined since the differential form
$e^{-i\Omega}[\nu(-i\partial)\varphi](\phi_G(m))$ has a compact support. The distribution $DH^G(M,\Omega,\nu)$ depends 
only of the class defined by $\nu$ in $H^*_G(M)$.

\subsection{Formal geometric quantization: definition}

The invariant scalar product on $\kgot^*$  provides an identification
$\kgot\simeq\kgot^*$.

\begin{defi}\label{defi:kir}
$\bullet$ The {\em Kirwan vector field} associated to $\phi_G$ is defined by
\begin{equation}\label{eq-kappa}
    \kappa_{G}(m)= -\phi_G(m)\cdot m, \quad m\in M.
\end{equation}

$\bullet$ We denote by $Z_G$ the set of zeroes of $\kappa_{G}$. Thus $Z_G$ is a $G$-invariant closed subset of $M$.
\end{defi}

The set $Z_G$, which is not necessarily smooth, admits the following description. Consider the closed Weyl chamber
$\mathfrak{t}^*_{\geq 0}$.
We see that
\begin{equation}\label{eq=Z-Phi-beta}
Z_G=\coprod_{\gamma\in\Bcal_G} Z_\gamma
\end{equation}
where $Z_\gamma$ corresponds to the compact set $G(M^\gamma\cap\phi_G^{-1}(\gamma))$, and $\Bcal_G=\phi_G(Z_G)\cap \tgot^*_{\geq 0}$.
The properness of $\phi_G$ insures that, for any compact subset $C\subset \tgot^*$, the intersection $\Bcal_G\cap C$ is finite.
Here $M^{\gamma}$ is the set of zeroes of the vector field on $M$ defined by the infinitesimal action of $\gamma$.

Let $\Scal\to M$ be the spinor bundle on $M$. The principal symbol of the Dirac operator $D_\Scal$ is the bundle map
$\sigma(M)\in \Gamma(\T^* M, \hom(\Scal^+,\Scal^-))$ defined by the Clifford action
$$\sigma(M)(m,\nu)=\clif_{m}(\tilde{\nu}): \Scal\vert_m^+\to \Scal\vert_m^-$$
where $\nu\in \T^* M\simeq \tilde{\nu}\in \T M$ is an identification associated to an invariant Riemannian metric on $M$.

\begin{defi}\label{def:pushed-sigma}
The symbol  $\sigma(M,\phi_G)$ shifted by the vector field $\kappa_{G}$ is the
symbol on $M$ defined by
$$
\sigma(M,\phi_G)(m,\nu)=\sigma(M)(m,\tilde{\nu}-\kappa_G(m))
$$
for any $(m,\nu)\in\T^* M$.
\end{defi}

For any $G$-invariant open subset $\Ucal\subset M$ such that $\Ucal\cap Z_G$ is compact in $M$, we see that the restriction
$\sigma(M,\phi_G)\vert_\Ucal$ is a transversally elliptic symbol on $\Ucal$. If $E\to M$ is a complex vector bundle, we denote by 
$D_{E,\phi_G}^\Ucal$ a pseudo-differential operator whose principal symbol defines the same class than $\sigma(M,\phi_G)\otimes E\vert_\Ucal$ in the group of equivariant $K$-theory of $\T^*_G\Ucal$ (see \cite{Atiyah74}).

Thus we can define the following localized equivariant indices (see \cite{Atiyah74,pep-vergne:witten}).

\begin{defi}\label{def:indice-localise}
$\bullet$ A closed invariant subset $Z\subset Z_G$ is called a {\em component} of $Z_G$ if it is a union of connected components of $Z_G$.

$\bullet$ If $Z$ is a {\em compact component} of $Z_G$, and $E\to M$ is an equivariant vector bundle, we denote by
$$
\QS_G(M,E,Z)\ \in\ \hat{R}(K)
$$
the equivariant index of the transversally elliptic operator $D_{E,\phi_G}^\Ucal$. Here $\Ucal$ is an invariant neighborhood of $Z$
so that $\Ucal\cap Z_G=Z$.
\end{defi}

By definition, $Z=\emptyset$ is a component of $Z_\Scal$ and $\QS_G(M,\Scal,\emptyset)=0$. For any $\gamma\in\Bcal_\Scal$, $Z_\gamma$ is a
compact component of  $Z_\Scal$.

When the manifold $M$ is compact, the set $\Bcal_G$ is finite and we have the decomposition
$\QS_G(M,\Lcal^k )=\sum_{\gamma\in\Bcal_G}\QS_G(M,\Lcal^k ,Z_\gamma)$.

\begin{definition}
When the moment map $\phi_G$ is proper, we define the formal geometric quantization of the data $(M,\phi_G,\Lcal^k)$ as
\begin{equation}\label{eq:formal-def}
\qfor_G(M,\Lcal^k):=\sum_{\gamma\in\Bcal_G}\QS_G(M,\Lcal^k,Z_\gamma).
\end{equation}
The sum of the right hand side is not necessarily finite but it converges in $\hat{R}(G)$
(see \cite{pep:formal2,MZ14,Hochs-Song:duke,pep:formal3}).
\end{definition}

In the following example, for any $\lambda\in\Z$, we denote by $\C_{[\lambda]}$ the vector space $\C$ with the action of
$S^1$ : $t\cdot z= t^\lambda z$, for $(t,z)\in  S^1\times \C_{[\lambda]}$.

\begin{example}\label{example:C}
Consider the $S^1$-manifold $M=\C_{[2]}$. The $S^1$-equivariant spinor bundle on
$M$ is $\Scal=M\times \left(\C_{[-1]}\oplus \C_{[1]}\right)$. 

Fix $a\in\N$.
Consider the equivariant line bundle $\Lcal(a)=M\times \C_{[a]}$ with connection $\nabla=d-\frac{i}{2} Im (zd\bar{z})$.
The two-form is $\Omega^a=\frac{1}{2} Im (dzd\bar{z})$, the moment map $\phi^a_{S^1}(z)=a+|z|^2$ is proper, and the corresponding critical set $Z^a_{S^1}$ is reduced to $\{0\}\subset M$.

A small computation shows that $\qfor_{S^1}(M,\Lcal(a)^k)=\QS_{S^1}(M,\Lcal(a)^k,\{0\})$ is equal to
$\C_{[ka]}\otimes\sum_{j\geq 0}\C_{[2j+1]}$ (see \cite{pep-vergne:witten}).
\end{example}

Let $m_G(\lambda,k)$ be the multiplicity of $V_\lambda$ in $\qfor_G(M,\Lcal^k)$. In other words,
$$
\qfor_G(M,\Lcal^k)=\sum_{\lambda\in A_G}m_G(\lambda,k)\,V_\lambda.
$$
In this context, the multiplicities $m_G(\lambda,k)$ have still an interpretation in terms of reduced spaces  (see \cite{Hochs-Song:duke}).
Hence, when the generic infinitesimal stabilizer is abelian, Theorem \ref{theo:QR-multiplicities} still holds, so $m_G(\lambda,k)\neq 0$ only if
$\lambda/k\in \phi_G(M)$.

 As in the previous section, we are interested in the asymptotic behavior of the following family of distributions
$$
\Theta^{(M,\CL)}_k := k^r\sum_{\lambda\in A_G} m_G(\lambda,k)\beta_{\lambda/k},
$$
where $r=\dim (G/T)/2$. For any test function $\varphi$ with compact support, the identity
$$
\ll\Theta^{(M,\CL)}_k,\varphi\rr := k^r\sum_{\lambda\in A_G} m_G(\lambda,k)\ll \beta_{\lambda/k},\varphi\rr
$$
is well defined since there exists only a finite number of terms such that $\ll \beta_{\lambda/k},\varphi\rr\neq 0$.

 Let us recall that we can associate the twisted Duistermaat-Heckman
 distribution $DH^G(M,\Omega, \Ahat_n(M))\in\Dcal'(\g^*)^G$
 to the equivariant form $\Ahat_n(M)$ for any $n\geq 0$ (see (\ref{eq:DH-noncompact})).

The 	aim of  this section is to prove the following extension of Theorem \ref{asympt1}.

\begin{theorem}\label{asympt1-noncompact}
Let $M$ be a spin even dimensional oriented manifold with a $G$-equivariant line bundle $(\CL,\nabla)$.
Suppose that the moment map $\phi_G$ is proper. When $k$ tends to $\infty$, we have the asymptotic expansion
\begin{equation}\label{eq:asymptotics-theta-k-noncompact}
\Theta^{(M,\Lcal)}_k \equiv \  j_\g^{1/2}(i\partial/k)\Big(k^d\sum_{n=0}^{\infty}
   \frac{1}{k^n} DH^G(M,\Omega, \Ahat_n(M))\Big) .
\end{equation}
\end{theorem}

\medskip

\begin{example}\label{example:C-bis}
We continue Example \ref{example:C} with $M=\C_{[2]}$ and $\Lcal(a)= M\times \C_{[a]}$. Then
$\Theta^{(M,\CL(a))}_k= \sum_{j\geq 0}\delta_{a+(2j+1)/k}$. The equivariant $\hat{A}$-class is
$\widehat{A}(M)(X)= \frac{X}{\sin(X)}$ for $X\in {\rm Lie}(S^1)\simeq \R$. Identity
(\ref{eq:asymptotics-theta-k-noncompact}) says that
$$
\Theta^{(M,\CL(a))}_k\equiv \frac{k}{2}\frac{i\partial/k}{\sin(i\partial/k)} 1_{[a,\infty[}.
$$
This is the formula given in Example \ref{ex:cohen}.
\end{example}

With the same proof than Theorem  \ref{asympt1-noncompact}, we obtain the following theorem.
\begin{theorem}\label{asympt1-noncompact-g}
Let $M$ be a spin even dimensional oriented manifold with a $G$-equivariant line bundle $\CL$.
Suppose that the moment map $\phi_G$ is proper. Let $g\in G$ of finite order. When $k$ tends to $\infty$, we have the asymptotic
equivalence
$$\res_{\g(g)^*}(k)\Big(\sum_{\lambda\in \Lambda} m(\lambda,k) \beta(g,\lambda)\Big)\equiv$$
 $$ j_{\g(g)}^{\frac{1}{2}}(i\partial/k)\det{}_{\g/\g(g)}^{\frac{1}{2}}(1-g^{-1} e^{-i\partial/k})
 \sum_{n=0}^{\infty} \frac{1}{k^{n}}  k^{\frac{\dim M^g}{2}}DH^{G(g)}(M^g,\Omega_g, V_g(n,k)).
 $$

\end{theorem}

We prove Theorem \ref{asympt1-noncompact} in the next two subsections.

\subsection{Formal geometric quantization: delocalized formulas}

We consider the generalized character
$$
\QS_G(M,\Lcal^k,Z_\gamma)=\sum_{\lambda\in A_G}\,m_\gamma(\lambda,k)\,V_\lambda.
$$
Here $\QS_G(M,\Lcal^k,Z_\gamma)(g)=\sum_{\lambda\in A_G}\,m_\gamma(\lambda,k)\,\chi_\lambda(g)$
defines a distribution on $G$ because the multiplicity function $\lambda\mapsto m_\gamma(\lambda,k)$ has at most a polynomial growth.

The function  $m_\gamma(\lambda,k)$
is in fact with at most polynomial growth in both variables
$(\lambda,k)$.
This follows for example from the general multiplicity formula for transversally elliptic operators given in
\cite{vergne:Izvestiya}. In this article,
the multiplicity formula is given for a single general transversally elliptic symbol
$\sigma$   and multiplicities are obtained as values on $A_G$ of a certain piecewise quasi-polynomial function on $A_G$.
If we consider a family   $\sigma_k=\sigma \otimes \CL^k$,
the formula is locally piecewise polynomial on a certain finite number of affine
cones  in $\t^*_{\geq 0} \oplus \R$ intersected with $A_G\oplus \Z$.


We consider the
corresponding distribution on $\g$
$$
\QS_G(M,\Lcal^k,Z_\gamma)(e^X)=\sum_{\lambda\in A_G}\,m_\gamma(\lambda,k)\,\chi_\lambda(e^X).
$$
The distribution on $\ggot^*$ defined by
$$
\Theta^\gamma_k:=k^{r}\sum_{\lambda\in A_G}\,m_\gamma(\lambda,k)\, \beta_{\lambda/k}
$$
is the Fourier transform of $\QS_G(M,\Lcal^k,Z_\gamma)(e^{X/k})j_\ggot^{1/2}(X/k)$. It is tempered also because
the multiplicity function $\lambda\mapsto m_\gamma(\lambda,k)$ has at most a polynomial growth.

In order to give a formula for the asymptotic of $\Theta^\gamma_k$, we introduce an equivariant form $P_\gamma(X)$
with generalized coefficients on $M$ that is supported in a small neighborhood of $Z_\gamma$.

Take $\chi:M\to\R$ a $G$-invariant function equal to $1$ in a neighborhood of
$Z_\gamma$, compactly supported,  and such that ${\rm Support}(\chi)\cap Z_G= Z_\gamma$. Consider the invariant
$1$-form $\theta:=(\kappa_G,-)$ where $\kappa_G$ is the Kirwan vector field associated to $\phi_G$, and $(-,-)$ is an invariant Riemannian metric.

 We consider the equivariant form $D\theta(X)= d\theta-\langle\Phi_\theta,X\rangle$, and the equivariant form with generalized coefficients
$$
P_\gamma(X):=\chi+i\, (d\chi)\theta\int_{0}^\infty e^{-itD\theta(X)}dt.
$$

Thus, if $f(X)$ is in the Schwartz space of $\g$, the integral
$$
\int_\g  P_\gamma(X)f(X)dX:=\chi \hat{f}(0)+i\, (d\chi)\theta\int_{0}^\infty e^{-it d\theta} \hat{f}(t\Phi_\theta)dt
$$
is convergent.
Indeed $\ll \Phi_\theta,\phi_G\rr(m)=\|\kappa_G(m)\|^2$ and so $\Phi_\theta\neq 0$ on the support of $d\chi$.
It defines a differential form with compact support since it is equal to $0$ outside the support of $\chi$.
So if $f(X)$ is a function from $\g$ to differential forms on $M$
such that $X\to f(X)$ is rapidly decreasing, then we can define
$\int_M  \int_\g P_\gamma(X)f(X) dX$.
Let $\varphi$ be a test function on $\g^*$,
and $\nu(X)$  an equivariant form.
Then $e^{-i\Omega(X)}\nu(X)\widehat{\varphi}(X)$ is rapidly decreasing
if $\nu(X)$ is an equivariant form with polynomial coefficients.
So we define
 $DH^G(M,\Omega,\nu P_\gamma)\in \Dcal'(\ggot^*)^G$
 by
 $$
 \ll DH^G(M,\Omega,\nu P_\gamma),\varphi\rr=
\frac{1}{(-2i\pi)^{d}}\int_M \int_\g e^{-i\Omega(X)}\nu(X) P_\gamma(X){\widehat{\varphi}}(X)dX.
$$

An improved version of Witten non abelian localization theorem is the following formula
\begin{equation}\label{eq:pepnonob}
 DH^G(M,\Omega,\nu)=\sum_{\gamma\in \CB_G} DH^G(M,\Omega,\nu P_\gamma)
 \end{equation}
which is somehow a consequence of the fact that $\sum_{\gamma\in \CB_G} P_\gamma$ is equal to $1$ in cohomology
(see \cite{pep:topology}).

\medskip

We will prove the following theorem in the next section.
\begin{theo}\label{theo-asymptotics-beta}
We have the following relation
$$
\Theta^\gamma_k\equiv k^{d}\, j_{\ggot}^{1/2}(i\partial/k)\sum_{n=0}^{\infty} \frac{1}{k^n}DH^G(M,\Omega,\hat{A}_n(M) P_\gamma).
$$
\end{theo}

We end this section by some observations.

We see that $P_\gamma(X)= \chi+\sum_{j=1}^{d} \,\omega_j \,\int_{0}^\infty t^{j-1} e^{it\langle\Phi_\theta,X\rangle}dt$
where $\omega_j$ is a differential form, compactly supported, of degree $2j$. Then
$$
E(X,k):=P_\gamma(X/k)= \chi+\sum_{j=1}^{d} \, k^j\,\omega_j \,\int_{0}^\infty t^{j-1} e^{it\langle\Phi_\theta,X\rangle}dt,\quad k\geq 1,
$$
is a polynomial in $k$ with value equivariant forms with generalized coefficients.

Let $U_\gamma\subset M$ be a relatively compact invariant subset containing the support of $\chi$. So we can choose
$r_\gamma>0$ so that $\hat A(M)(X)$ is well defined on $U_\gamma$ when $\|X\|<r_\gamma$.

We will deduce Theorem \ref{theo-asymptotics-beta} from the following result proved in \cite{pep-vergne:bismut}.

\begin{prop}\label{pro:ParVER}
If $\|X\|<r_\gamma$, we have the relation
$$
\QS_G(M,\Lcal^k,Z_\gamma)(e^X)=\frac{1}{(-2i\pi)^{d}}\int_M \hat{A}(M)(X)\, P_{\gamma}(X)\, e^{-ik\Omega(X)}.
$$
\end{prop}

So, for $k\geq 1$, the generalized function $\QS_G(M,\Lcal^k,Z_\gamma)(e^{X/k})j_\g^{1/2}(X/k)$ coincides with
$$
\frac{1}{(-2i\pi)^{d}}\int_M j_\g^{1/2}(X/k)\hat{A}(M)(X/k)\, P_{\gamma}(X/k)\, e^{-ik\Omega(X/k)}
$$
on the ball $\{\|X\|/k< r_\gamma$\}.

Let us compute the Laurent series of 
$$[j_\g^{1/2}(X/k)\,\hat{A}(M)(X/k)\, P_{\gamma}(X/k)\, e^{-ik\Omega(X/k)}]_{[2d]}$$
where $[-]_{[2d]}$ means the component of
maximal degree $2d$ in $\Acal^\bullet(M)$.
We write $j_\g^{1/2}(X) \Ahat(M)(X)=
\sum_{n=0}^{\infty}d_n(X)$
as a sum of closed equivariant forms of equivariant degree $n$.
Recall that $E(X,k)=P_{\gamma}(X/k)$ and $b(X,k)=e^{-ik\Omega(X/k)}$
 depend polynomially of $k$.

We write also $j_\g^{1/2}(X/k)\hat{A}(M)(X)=\sum_{n=0}^{\infty}\tau_n(X)$ where $\tau_n(X)$
is an equivariant form with  coefficients homogeneous polynomial functions of $X$  of degree $n$.
We proceed as in the proof of Lemma \ref{crucial}
comparing the terms of top exterior degree $2d$.
Arguing as in the proof of Lemma \ref{lem:gradedstupidcompact}, we have the following formula.

\begin{lemma}\label{lem:stupidgradingMnoncompact}
$$\Big[b(X,k)\Big(\sum_{n=0}^{\infty} \frac{1}{k^n}\tau_n(X)\Big)\, E(X,k)\Big]_{[2d]}=
\Big[e^{-i\Omega(X)}\Big(\sum_{n=0}^{\infty} \frac{1}{k^n}d_n(X)\Big)\, P_\gamma(X)\, \Big]_{[2d]}.$$
\end{lemma}

Thus we can write formally
$$\QS_G(M,\Lcal^k,Z_\gamma)(e^{X/k})j_\g^{1/2}(X)=k^{d}\sum_{n=0}^{\infty}\,\frac{1}{k^n} g_n(X)$$
where $g_n(X)$ is the distribution
$$
g_n(X):=\frac{1}{(-2i\pi)^{d}}\int_M d_n(X)\, P_\gamma(X)\, e^{-i\Omega(X)}.
$$

\subsection{Proof of Theorem \ref{theo-asymptotics-beta}}
We fix a $G$-invariant function $h:\g\to\R$ equal to $1$ in a neighborhood of $0$ and with support contained in $\|X\|<r_\gamma$.
If $\varphi$ is a smooth function with compact support on $\g^*$, we have
$\ll\Theta^\gamma_k,\varphi\rr= \mathrm{I}_k+ \mathrm{J}_k$ where
$$
\mathrm{I}_k=\int_{\g}\QS_G(M,\Lcal^k,Z_\gamma)(e^{X/k})j_\g^{1/2}(X/k)h(X/k)\widehat{\varphi}(X)dX
$$
and
$$
\mathrm{J}_k=\int_{\g}\QS_G(M,\Lcal^k,Z_\gamma)(e^{X/k})j_\g^{1/2}(X/k)(1-h(X/k))\widehat{\varphi}(X)dX.
$$

For estimating  $\mathrm{J}_k$, we will use estimation of  the Fourier coefficients $m_\gamma(\lambda,k)$ and we will prove that $\mathrm{J}_k=O(k^{-\infty})$.
For estimating $\mathrm{I}_k$,
we will use  Proposition (\ref{pro:ParVER}).

To analyze these expressions, we  need the following technical lemma.

\begin{lemma}\label{prop:majorationschiantes}
Let $u(X)$ be a $C^{\infty}$ function of $X$.
Assume that $u(X)$ and all its derivatives are  of at most polynomial growth and
  that $u(X)$ vanishes at order $N$ at $X=0$.

Let $R$ be an integer.
  Then there exists a constant $c_{R}$ such that
$$
|\int_\g  u(X/k)
e^{i\langle \xi,X\rangle} \widehat{\varphi}(X)dX|\leq \frac{1}{k^{N+1}}\frac{c_{R}}{(1+\|\xi\|^2)^R},
$$
for all $(\xi,k)\in\ggot^*\times \N\setminus\{0\}$.
\end{lemma}

\begin{proof}
We may write in multi-index notation
$$u(X)=\sum_{\alpha,|\alpha|=N+1} X^{\alpha}v_\alpha(X)$$
with  $v_\alpha(X)$ (and derivatives)  bounded by  polynomial functions of $X$.
We  compute $I_R= k^{N+1}(1+\|\xi\|^2)^R\int_\g  u(X/k)
e^{i\langle \xi,X\rangle} \widehat{\varphi}(X)dX$. Thus
$$
I_R=(1+\|\xi\|^2)^R\int_\g \sum_{|\alpha|=N+1} X^{\alpha} v_\alpha(X/k)
e^{i\langle \xi,X\rangle} \widehat{\varphi}(X)dX.
$$

Let $\Delta=-\sum_{j}\partial_{X_j}^2$ be the Laplacian on $\g$.
Then $I_R$ is equal to
$$
 \sum_{|\alpha|=N+1} \int_\g X^{\alpha} v_\alpha(X/k)
\widehat{\varphi}(X) \left((1+\Delta^2)^R\cdot (e^{i\langle \xi,X\rangle}\right)dX=
\int_\g  L_R(X,k)e^{i\langle \xi,X\rangle}dX$$
where
$L_{R}(X,k)=(1+\Delta^2)^R \cdot
\left(\sum_{|\alpha|=N+1}X^{\alpha} v_\alpha(X/k) \widehat{\varphi}(X)\right).$

Using the fact that $\widehat{\varphi}(X)$ is rapidly decreasing (as well as all its derivatives), that the derivatives
of the functions $v_\alpha$ are bounded by polynomials,
and  that  $1/k\leq 1$ for $k\geq 1$,
we see that $L_{R}(X,k)$ can be bounded by a  rapidly decreasing function of $X$
independent of $k$.
So
$|\int_\g  L_R(X,k) e^{i\langle \xi,X\rangle}dX|\leq c_R.$
\end{proof}

We return to our proof.
We start  by checking  that $\mathrm{J}_k=O(k^{-\infty})$. For this computation, we can assume that
$\varphi$ is $G$-invariant.
Let
$$
c(\lambda,k)=\int_{\g}\chi_\lambda(e^{X/k})j_\g^{1/2}(X/k)(1-h(X/k))\widehat{\varphi}(X)dX.
$$

This
is equal to $\vol(G\lambda)\int_{\g}e^{i\langle\lambda/k,X\rangle}(1-h(X/k))\widehat{\varphi}(X)dX$.

By definition $\QS_G(M,\Lcal^k,Z_\gamma)(e^X)=\sum_{\lambda\in A_G}\,m_\gamma(\lambda,k)\,\chi_\lambda(e^X)$, so  we get
$$
\mathrm{J}_k=\sum_{\lambda\in A_G}\,m_\gamma(\lambda,k)\vol(G\lambda)
\int_{\g}e^{i\langle\lambda/k,X\rangle}(1-h(X/k))\widehat{\varphi}(X)dX.
$$

Now consider $u(X)=1-h(X)$. It vanishes identically in a neighborhood of $0$.
As $h(X)$ is compactly supported, $u(X)$ and all derivatives are bounded.
So we can apply Lemma   \ref{prop:majorationschiantes} and obtain
$|c(\lambda,k)|\leq   \vol(G\lambda)\frac{1}{k^{N+1}} \frac{c_{R,N}}{(1+\|\lambda/k\|^2)^R}$
for any integers $(R,N)$.
Since the multiplicities $m_\gamma(\lambda,k)$ have at most a polynomial growth in the variable $(\lambda,k)$,
 we can conclude
that  $\mathrm{J}_k=O(k^{-\infty})$.

We now estimate
$$\mathrm{I}_k=\int_{\g}\QS_G(M,\Lcal^k,Z_\gamma)(e^{X/k})j_\g^{1/2}(X/k)h(X/k)\widehat{\varphi}(X)dX.
$$
As $h(X/k)=0$ when $\|X\|/k\geq r_\gamma$, we may use Proposition \ref{pro:ParVER}.
Let $\tau^h(X)=h(X)\hat{A}(M)(X)j_\g^{1/2}(X)$,
a smooth compactly supported function from $\g$ to  differential forms on $U_\gamma$.
We fix $x\in U_\gamma$ and consider
$\tau^h_x(X)$, a function from $\g$ to
$\Lambda T^*_xM$.
Define
$$\mathrm{I}^x_k=\int_{\g}\tau^h_x(X/k) E_{x}(X,k) b_x(X,k)\widehat{\varphi}(X)dX$$
with
$$E_{x}(X,k)= \chi(x)+
\sum_{j=1}^{d} \, k^j\,(\omega_{j})|_x \,\int_{0}^\infty t^{j-1} e^{it\langle\zeta,X\rangle}dt$$
if $\zeta=\Phi_\theta(x)$
and $b_x(X,k)=e^{-ik\Omega_x}e^{i \ll \xi,X\rr}$ if $\xi=\phi_G(x)$. So
$$
{\rm I}^x_k=\int_M \int_{\g}\tau^h_x(X/k) E_{x}(X,k) b_x(X,k)\widehat{\varphi}(X)dX.
$$

In view of Lemma \ref{lem:stupidgradingMnoncompact},
we only need to prove that the asymptotic expansion of $\mathrm{I}^x_k$ is obtained by replacing
$\tau^h_x(X)$ by its Taylor series (we need to care of uniform estimates in $x$ in the compact support of $\chi$).

Thus we fix $x$ and write $E_x(X,k)=\chi(x)+R_x(X,k)$.
So
$\mathrm{I}^x_k=\mathrm{S}^x_k+\mathrm{T}^x_{k}$
with
$$\mathrm{S}^x_k=\chi(x)  \int_{\g}\tau_x^h(X/k) b_x(X,k)\widehat{\varphi}(X)dX$$
and
$$\mathrm{T}^x_{k}=\int_{\g}R_x(X,k)\tau_x^h(X/k) b_x(X,k)\widehat{\varphi}(X) dX.$$

The expression  for $\mathrm{S}^x_k$ has been already analyzed when proving Proposition \ref{asympt1}, and
indeed it admits an asymptotic expansion obtained by replacing $\tau^h_x(X)$ by its Taylor series.

Consider $a(X,t)=e^{i\ll\xi,X\rr}e^{it \ll\zeta,X\rr} P(t)$
where $P(t)$ is a polynomial function on $t$ with values in $\Lambda T^*_xM$.
We only need to prove that 
$$
\mathrm{W}^x_k=\int_{t=0}^{\infty}\int_\g a(X,t)\tau^h_x(X/k) \widehat{\varphi}(X) dX
$$
 admits an asymptotic expansion, also obtained by replacing $\tau^h_x(X)$ by its Taylor series.
 We fix $x\in M$, write
 $\tau^h_x(X)=\tau_{\leq N}(X)+\tau_{>N}(X)$.
Then $\tau_{> N}(X)$ vanishes  at order $N$ at $X=0$.
 As $\tau^h(X)$ was compactly supported, $\tau_{\leq N}(X)$ and $\tau_{>N}(X)$
 are of at most polynomial growth, as well as derivatives.
 So we use Lemma \ref{prop:majorationschiantes}
 and obtain, for any positive  integer $R$,
 $$
 \vert\int_\g a(X,t)\tau_{>N}(X/k) \widehat{\varphi}(X) dX \vert \leq P(t) \frac{1}{k^{N+1}}
 \frac{c_{N,R}}{(1+|\xi+t\zeta|^2)^R}\cdot
 $$

If $R$ is sufficiently large, $\int_{t=0}^\infty P(t) \frac{c_{N,R}}{(1+|\xi+t\zeta|^2)^R}<\infty$.
So we obtain our estimate for the rest. All our estimates can be done uniformly in $x$ when $x$ runs in the compact support of $\chi$.
This ends the proof of
Theorem \ref{theo-asymptotics-beta}.

\subsection{Proof of Theorem \ref{asympt1-noncompact}}

The decomposition $\Theta^{(M,\Lcal)}_k=\sum_\gamma \Theta^\gamma_k$ is well-defined since
the distribution $\Theta^\gamma_k$ is supported in $\{\|\xi\|\geq \|\gamma\|\}$ (see \cite{pep:formal2}).
If $\varphi$ is a test function, then
$$
\ll\Theta^{(M,\Lcal)}_k,\varphi\rr= \sum_\gamma \ll\Theta^\gamma_k,\varphi\rr
$$
where the sum in the right hand side has only a finite number of non-zero terms.

Thanks to Theorem \ref{theo-asymptotics-beta}, we have the asymptotic expansion
$$
\Theta^\gamma_k\equiv k^{d}\, j_{\ggot}^{1/2}(i\partial/k)\sum_{n=0}^{\infty}\frac{1}{k^n}DH^G(M,\Omega,\hat{A}_n(M) P_\gamma)
$$
for any $\gamma$. Hence
$\Theta^{(M,\Lcal)}_k\equiv k^{d}\, j_{\ggot}^{1/2}(i\partial/k)\sum_{n=0}^{\infty}k^{-n}\theta_n$
with $\theta_n$ equal to $\sum_\gamma DH^G(M,\Omega,\hat{A}_n(M) P_\gamma)$.
The proof of Theorem \ref{asympt1-noncompact} is complete since
$\sum_\gamma DH^G(M,\Omega,\hat{A}_n(M) P_\gamma)= DH^G(M,\Omega,\hat{A}_n(M))$  by Equation (\ref{eq:pepnonob}).

\section{Functoriality}

\subsection{$H$ and $G$}

Let $H$ be a connected compact subgroup of $G$.
Let $r:\g^*\to \h^*$ be the projection.
For  $\theta$ a distribution with compact support on $\g^*$,
the push-forward $ \ll r_*\theta,\varphi\rr = \ll \theta, \varphi\circ r\rr $
is well defined.  The Fourier transform
$\Fcal_{\h^*}(r_*\theta)$ is the restriction to $\h$ of the Fourier transform
$\Fcal_{\g^*} (\theta)$.
We can define $r_*\theta$ more
generally whenever $\theta$ is compactly supported along the fibers of $r$.

The twisted Duistermaat-Heckman distributions behave very well under the push-forward map $r_*$.


From (\ref{eq:DH-noncompact}) we get immediately the following

\begin{proposition}Suppose that the moment map $\phi_H:=r\,\circ\, \phi_G$ is proper. Then, for any $\nu\in H_G^*(M)$, we have
$$
r_*DH^G(M,\Omega,\nu)=DH^H(M,\Omega,\nu).
$$
\end{proposition}

Re-scaling behave also very well under the push-forward map $r_*$:
$$
\res_{\h^*}(k)\circ r_*=r_*\circ \res_{\g^*}(k).
$$

\bigskip

In the rest of the article, objects associated to $\lambda$ are associated to $G$, while objects associated to $\mu$ are associated to $H$.

\medskip

For $\mu\in A_H$ and $\lambda\in A_G$, let
 $c(\mu,\lambda)$ be the multiplicity of the representation
 $V_\mu$ of $H$ in the restriction of $V_\lambda$ to $H$.
 In other words, for $h\in H$,
$\chi_\lambda(h)=\sum_{\mu\in A_H}\, c(\mu,\lambda) \chi_\mu(h)$.

Consider the $H$-invariant function $j_{\g/\h}^{1/2}(Y)=\det_{\g/\h}^{1/2}(\frac{e^{Y/2}-e^{-Y/2}}{Y})$ on $\h$.
Its Fourier transform is a compactly supported measure $B_{\g/\h}$ on $\h^*$.
We have an exact relation between the push-forward of the measure $\beta_\lambda$ on $\h^*$ and measures of $H$-admissible coadjoint orbits. 
We denote by $B_{\g/\h}^k=\res_{\h^*}(k)B_{\g/\h}$.

\begin{lemma}\label{rescaled1orbits}
For $\lambda\in A_G$,
we have the relations
\begin{eqnarray*}
r_*(\beta_\lambda)&=&B_{\g/\h}\star\Big(\sum_{\mu\in A_H}\, c(\mu,\lambda)\,\beta_{\mu}\Big),\\
\res_{\h^*}(k)\left(r_*(\beta_\lambda)\right)&=&B^k_{\g/\h}\star\Big(\res_{\h^*}(k)
\big(\sum_{\mu\in A_H}\, c(\mu,\lambda)\,\beta_{\mu}\big)\Big).
\end{eqnarray*}
Here the $\star$ sign denotes the convolution.
\end{lemma}

\begin{proof}
The first identity follows immediately by Fourier transform of the formula:
$\chi_\lambda(e^Y)j_{\g}^{1/2}(Y)=j_{\g/\h}^{1/2}(Y)
\sum_{\mu\in A_H}\, c(\mu,\lambda)\,\chi_\mu(e^Y) j_{\h}^{1/2}(Y)$,  for $Y\in \h$.

We get the second identity by applying the operator $\res_{\h^*}(k)$ to the first one.
 \end{proof}

We now prove a similar descent formula. We write $\g=\h\oplus\q$ where $\q$ is $H$-invariant.
Let $g\in H$ and consider the function
$$
d(Y):=\chi_\lambda(ge^Y)j_{\g(g)}^{1/2}(Y)\det{}^{1/2}_{\g/\g(g)}(1-g^{-1} e^{-Y}),\quad Y\in \ggot(g).
$$

In one hand $d(Y)=\ll \beta(g,\lambda), e^{i\langle-,Y\rangle}\rr$ when $Y\in \g(g)$, and on the other hand we have
$$
d(Y)=j_{\g(g)/\h(g)}^{1/2}(Y)\det{}^{1/2}_{\q/\q(g)}(1-g^{-1}
e^{-Y})\sum_{\mu}c(\lambda,\mu)\ll \beta(g,\mu), e^{i\langle-,Y\rangle} \rr
$$
when $Y\in\hgot(g)$.

If $A$ is a compactly supported distribution on $\g(g)^*$, we still denote by $r_*(A)$ the push-forward distribution on $\h(g)^*$.
We denote by $B_{\g(g)/\h(g)}$ the compactly supported measure on $\h(g)^*$ which is the Fourier transform of the $H(g)$-invariant function $j_{\g(g)/\h(g)}^{1/2}$. We denote by $B_{\g(g)/\h(g)}^k:=\res_{\h(g)^*}(k)B_{\g(g)/\h(g)}$.
Let $C_{\q/\q(g)}$ be the Fourier transform of the $H(g)$-invariant function $Y\mapsto\det{}^{1/2}_{\q/\q(g)}(1-g^{-1}
e^{-Y})$. We denote by $C_{\q/\q(g)}^k:=\res_{\h(g)^*}(k)C_{\q/\q(g)}$.

The previous identities give the following result.

\begin{lemma}\label{rescaled1orbits-g}
Let $g\in H$. Then $\res_{\h(g)^*}(k)\Big(r_*\beta(g,\lambda)\Big)$ is equal to
$$
B_{\g(g)/\h(g)}^k\star C_{\q/\q(g)}^k\star\left(\res_{\h(g)^*}(k)
\Big(\sum_{\mu\in A_H} c(\mu,\lambda)\beta(g,\mu)\Big)\right).
$$
\end{lemma}

\medskip

\subsection{{Functoriality relatively to restrictions}}

Here we consider an oriented spin manifold $M$ of even dimension with an action of a compact connected Lie group $G$. Let $\Lcal\to M$ be a $G$-equivariant line
bundle equipped with an invariant Hermitian connection $\nabla$. We assume here that $\phi_G$ is proper.
In this case we can define $\Qfor_G(M,\Lcal^{k})\in \Rfor(G)$ for any $k\geq 1$.

The main result of this section is the following theorem.

\begin{theo}\label{theo:formal}
Suppose that $H\subset G$ is a closed connected subgroup such that  $\phi_H$ is proper. Then
\begin{enumerate}
\item $\Qfor_G(M,\Lcal^{k})$ is $H$-admissible,
\item $\Qfor_G(M,\Lcal^{k})\vert_H=\Qfor_H(M,\Lcal^{k})$.
\end{enumerate}
\end{theo}

\bigskip

We will prove Theorem \ref{theo:formal} under the assumption that the generic infinitesimal stabilizer of the $G$-action on $M$ is abelian.
It can be shown easily that it implies the general case (see Lemma 4.2 in \cite{pep:formal3}).

Let us denote by $c(\lambda,\mu)$ the multiplicity of $V_\mu$ in $V_\lambda\vert_H$. If we consider the generalized characters
$\Qfor_G(M,\Lcal^{k})=\sum_{\lambda\in\hat{G}}\,m_G(\lambda,k)\, V_\lambda$ and
$\Qfor_H(M,\Lcal^{k})=\sum_{\mu\in\hat{H}}\,m_H(\mu,k)\, V_\mu$,
we see that Theorem \ref{theo:formal} is equivalent to the following theorem.

\begin{theo}\label{theo:formal-2}
For any $\mu\in A_H$, we have
$$
m_H(\mu,k)=\sum_{\lambda\in A_G}\,m_G(\lambda,k)\, c(\lambda,\mu),
$$
where the right hand side is a finite sum.
\end{theo}

We consider the sequences of distributions
$\Theta^G_k =k^r\sum_{\lambda}\,m_G(\lambda,k)\, \beta_{\lambda/k}$ and
$\Theta^H_k = k^{r'}\sum_{\mu}\,m_H(\mu,k)\, \beta_{\mu/k}$ associated to the 
formal geometric quantizations
$\Qfor_G(M,\Lcal^{k})$ and $\Qfor_H(M,\Lcal^{k})$.

We start with the following proposition.

\begin{prop}
$\bullet$ We have $B^k_{\g/\h}\star \Theta^H_k\equiv j_{\g/\h}^{1/2}(i\partial/k)\Theta^H_k$.

$\bullet$ The push-forward $r_*(\Theta^G_k)$ is a well defined sequence of distributions on $\hgot^*$.
We have
$$
r_*(\Theta^G_k)\equiv j_{\g/\h}^{1/2}(i\partial/k)\Theta^H_k.
$$
\end{prop}
\begin{proof}
Since $m_H(\mu,k)$ is a piecewise quasi-polynomial function (see Section \ref{sec:quasi-polynomial-multiplicity}),
the first point is a consequence of Proposition \ref{prop:asymptotic-quasi-polynomial-star-B}.

\medskip

We now prove the second point. Let $\varphi\in\Ccal^{\infty}(\h^*)$ be a function with compact support. Let $R_\varphi>0$ such that the compact set
$\Kcal_\varphi:=\phi_G(\phi_H^{-1}({\rm Support}(\varphi))$ is contained in $\{\xi\in \g^*,\, \|\xi\|\leq R_\varphi\}$.
The expression
$$
\ll r_*(\Theta^G_k),\varphi \rr:= k^r\sum_{\lambda} m_G(\lambda,k)\,\varphi\left(r(\lambda/k)\right)
$$
is well-defined since the term $m_G(\lambda,k)\,\varphi\left(r(\lambda/k)\right)$ is non-zero only for a finite number of $\lambda$. More precisely,
since $m_G(\lambda,k)\neq 0$ only if $\lambda/k\in \phi_G(M)$, we see that $m_G(\lambda,k)\,\varphi\left(r(\lambda/k)\right)\neq 0$ only
if $\lambda/k\in \Kcal_\varphi$.

If $\epsilon\in\Ccal^{\infty}(\g^*)$ is a function with compact support such that $\epsilon(\xi)=1$ if $\|\xi\|\leq R_\varphi$, we define
$\tilde{\varphi}:= \varphi\circ r\times\epsilon $. We see that
$$
\ll r_*(\Theta^G_k),\varphi \rr = \ll\Theta^G_k, \tilde{\varphi} \rr
= k^{d}\sum_{n=0}^N\frac{1}{k^n}\ll \theta^G_n,\tilde{\varphi} \rr + o(k^{d-N})
$$
where $(\theta^G_n)$ is the family of distributions on $\g^*$ such that
$$
j_{\ggot}^{1/2}(i\partial/k)\sum_{n=0}^{\infty}\frac{1}{k^n}DH^G(M,\Omega,\hat{A}_n(M))=\sum_{n=0}^{\infty}\frac{1}{k^n} \theta^G_n.
$$
We check easily that, for any polynomial $P$ on $\g$, we have
$$
\ll DH^G(M,\Omega,\hat{A}_n(M)),P(\partial)\tilde{\varphi}\rr=
\ll DH^H(M,\Omega,\hat{A}_n(M)),P\vert_\h(\partial)\varphi\rr .
$$
Hence the formal series
$\ll j_{\ggot}^{1/2}(i\partial/k)\sum_{n=0}^{\infty}\frac{k^d}{k^n}DH^G(M,\Omega,\hat{A}_n(M)),\tilde{\varphi}\rr$
is equal to
$$
\ll j_{\g}^{1/2}(i\partial/k)\sum_{n=0}^{\infty}\frac{k^d}{k^n}DH^H(M,\Omega,\hat{A}_n(M)),\varphi \rr
\equiv \ll j_{\g/\h}^{1/2}(i\partial/k)\Theta^H_k,\varphi\rr .
$$
\end{proof}

Consider the function $m'_H(\mu,k):= \sum_{\lambda\in A_G}\,m_G(\lambda,k)\, c(\lambda,\mu)$
and the difference $d(\mu,k):=m_H(\mu,k)-m'_H(\mu,k)$. Our aim is to show that $d=0$.

\medskip

\begin{proposition}\label{coro:g=1}
We have $\res_{\h^*}(k)\Big(\sum_{\mu\in A_H} d(\mu,k) \beta_\mu \Big)\equiv 0$.
\end{proposition}
\begin{proof}  The previous proposition tells us that
$$
r_*(\Theta^G_k)\equiv B^k_{\g/\h}\star  \res_{\hgot^*}(k)\Big(\sum_{\mu\in A_H}m_H(\mu,k)\,\beta_{\mu}\Big).
$$
If we use Lemma \ref{rescaled1orbits}, we can compute $r_*(\Theta^G_k)$ in another manner:
\begin{eqnarray*}
r_*(\Theta^G_k) &=& \res_{\h^*}(k)\Big(\sum_{\lambda\in A_G} m_G(\lambda,k) r_*(\beta_\lambda)\Big)\\
&=&B^k_{\g/\h}\star \res_{\hgot^*}(k)\Big(\sum_{\mu\in A_H}m_H'(\mu,k)\,\beta_{\mu}\Big).
\end{eqnarray*}
At this stage we obtain that $B^k_{\g/\h}\star \res_{\h^*}(k)\big(\sum_{\mu\in A_H} d(\mu,k) \beta_\mu \big)\equiv 0$, or equivalently $B^k_{\g/\h}\star \big(\sum_{\mu\in A_H} d(\mu,k) \beta_{\mu/k} \big)\equiv 0$. Since the function
$d(\mu,k)$ is a piecewise quasi-polynomial function, we know that $\Theta_k:=\sum_{\mu\in A_H} d(\mu,k) \beta_{\mu/k}$ admits an asymptotic expansion
$$
\Theta_k\equiv k^{n_o}\sum_{n=0}^\infty k^{-n}\theta_n(k)
$$
where the distributions $\theta_n(k)$ depends periodically on $k$ (see Proposition \ref{prop:asymptotic-quasi-polynomial}). Thanks to Proposition \ref{prop:asymptotic-quasi-polynomial-star-B}, we know that $B^k_{\g/\h}\star \Theta_k$
admits the asymptotic expansion
$$
0\equiv B^k_{\g/\h}\star\Theta_k\equiv j_{\g/\h}^{1/2}(i\partial/k)
\Big(k^{n_o}\sum_{n=0}^\infty k^{-n}\theta_n(k)\Big).
$$
If we write the Taylor series $j_{\g/\h}^{1/2}(X)=\sum_{n=0}^\infty d_n(X)$, we obtain for any $n\geq 0$ the relation
$\sum_{m+l=n} d_m(i\partial)\theta_l(k)=0$. As $d_0=1$, we see that $\theta_n(k)=0$ for all $n\geq 0$.
\end{proof}

\medskip

If we use Lemma \ref{rescaled1orbits-g} and Theorem \ref{asympt1-noncompact-g}, we can prove similarly the following extension of Proposition \ref{coro:g=1}.

\begin{proposition}\label{proformal}
For any $g\in H$ of finite order, we have
$$
\res_{\h(g)^*}(k)\Big(\sum_{\mu\in A_H} d(\mu,k) \beta(g,\mu)\Big)\equiv 0.
$$
\end{proposition}

The function $d(\mu,k)$ is defined for $\mu\in A_H=(\rho_H+\Lambda_H)\cap \t_{H,>0}^*$. Let $\tilde{\Lambda}_H$ 
be the lattice generated by $\Lambda_H$ and $\rho_H$. Let $W_H$ be the Weyl group of $H$. We can then extend 
the function $d(\mu,k)$ to a $W_H$-anti-invariant function $\tilde d$ on $\tilde{\Lambda}_H$:
${\tilde d}(w\mu,k)=\epsilon(w) d(\mu,k)$ for $w\in W_H$, $\mu\in A_H$ and
${\tilde d}(\mu,k)=0$ if $\mu\notin \rho_H+\Lambda_H$.

Let $T_H$ be a Cartan subgroup of $H$. We have a covering $\tilde{T}_H\to T_H$ such that $\tilde{\Lambda}_H$ is the weight lattice of the torus
$\tilde{T}_H$. Take $\tilde{g}\in \tilde{T}_H$ and its image $g\in T_H$.

We now  identify $H(g)$-invariant distributions on $\h(g)^*$ to $W_{H(g)}$-anti-invariant distributions
on $\t_H^*$ through the isomorphism $\Rcal_{\hgot(g)}$. Using Lemma \ref{lem:Harbeta}, we have
$\Rcal_{\h(g)}\left(\beta(g,\mu)\right)=\gamma_{\tilde{g}} \sum_{w\in W_G} \epsilon(w) \tilde{g}^{w\mu}\delta_{w\mu}$.

We thus obtain the following proposition.

\begin{proposition}\label{proformalbis}
For any $\tilde{g}\in \tilde{T}_H$ of finite order, we have
$$
\sum_{\nu\in \tilde{\Lambda}_H} {\tilde d}(\nu,k) \, \tilde{g}^{\nu}\, \delta_{\nu/k}
\equiv 0.
$$
\end{proposition}

Since ${\tilde d}(\nu,k)$ is a piecewise quasi-polynomial function on $\tilde{\Lambda}_H\times\Z_{>0}$,
Proposition \ref{proformalbis} implies that $\tilde{d}=0$ (see Proposition 3.1 in
\cite{pep-vergne:asymptotics}). Hence $d=0$~: the proof of Theorem \ref{theo:formal} is complete.

\end{document}